\documentclass[oneside]{amsart}
\usepackage{amsmath,amsthm,amssymb,amsfonts,ifpdf,verbatim}
\usepackage{enumitem}
\usepackage{graphicx}
\usepackage[usenames,dvipsnames]{color}

\ifpdf
\usepackage[pdftex,pdfstartview=FitH,pdfborderstyle={/S/B/W 1},%
colorlinks=true, linkcolor=blue, urlcolor=blue, citecolor=blue]{hyperref}
\else
\usepackage[dvips]{hyperref}
\fi
\numberwithin{equation}{section}

\newtheorem {theorem}    {Theorem}[section]
\newtheorem {lemma}      [theorem]    {Lemma}

\newtheorem {corollary}  [theorem]    {Corollary}
\newtheorem {proposition}[theorem]    {Proposition}

\theoremstyle{definition}
\newtheorem {definition} [theorem]    {Definition}

\newtheorem {remark}    [theorem]    {Remark}

\newcounter{AbcT}

\numberwithin{equation}{section}

\newcommand {\equ}[1]     {\eqref{#1}}

\newcommand {\cond}   {\:|\:}

\newcommand {\A} {{\mathbb A}}

\newcommand {\N} {{\mathbb N}}
\newcommand {\Q} {{\mathbb Q}}
\newcommand {\R} {{\mathbb R}}

\newcommand {\T} {{\mathbb T}}
\renewcommand {\H} {{\mathbb H}}
\newcommand {\Z} {{\mathbb Z}}
\newcommand {\cP} {{\mathcal P}}

\DeclareMathOperator{\supp}{supp}

\DeclareMathOperator{\SL}{SL}

\DeclareMathOperator{\PGL}{PGL}

\newcommand {\IGNORE}[1]  {}

\newcommand {\absolute}[1] {\left| {#1} \right|}

\newcommand{\nope} {}

\newcommand{\G}{\mathbf{G}}
\newcommand{\GG}{\mathbb{G}}

\newcommand{\Zd}{\mathbb{Z}^d}
\newcommand{\n}{\mathbf{n}}
\newcommand{\ba}{\mathbf{a}}

\newcommand{\urad}{\UU_{\textrm{rad}}}

\newcommand\vol{\operatorname{vol}}
\newcommand\md{\operatorname{mod}}
\newcommand\fp{\operatorname{\mathfrak{p}}}
\newcommand\fg{\operatorname{\mathfrak{g}}}

\newcommand\fl{\operatorname{\mathfrak{l}}}
\newcommand\Hh{\operatorname{H}}
\newcommand\cA{\mathcal{A}}

\newcommand\Ad{\operatorname{Ad}}
\newcommand{\LL}{\mathbb{L}}
\newcommand{\HH}{\mathbb{H}}
\newcommand{\FF}{\mathbb{F}}
\newcommand{\UU}{\mathbb{U}}
\newcommand{\VV}{\mathbb{V}}
\newcommand{\ZZ}{\mathbb{Z}}
\newcommand{\MM}{\mathbb{M}}

\begin{document}

	\title[Joinings of higher-rank torus actions]{Joinings of higher rank torus actions on homogeneous spaces}
	\author[M. Einsiedler]{Manfred Einsiedler}
	\address[M. E.]{ETH Z\"urich, R\"amistrasse 101
	CH-8092 Z\"urich
	Switzerland}
	\email{manfred.einsiedler@math.ethz.ch}
	\author[E. Lindenstrauss ]{Elon Lindenstrauss}
	\address[E. L.]{The Einstein Institute of Mathematics\\
	Edmond J. Safra Campus, Givat Ram, The Hebrew University of Jerusalem
	Jerusalem, 91904, Israel}
	\email{elon@math.huji.ac.il}
	\thanks{M.~E.~acknowledges the support by the SNF (Grant 200021-127145 and 200021-152819).
	E.~L.~acknowledges the support of the ERC (AdG Grant 267259), the Miller Institute and MSRI. The authors gratefully acknowledge the support of
	the Israeli Institute for Advanced Studies at the Hebrew University, where a good portion
	of this work was carried out under ideal working conditions.}
	\date{February 2015}
	\begin{abstract}
	We show that joinings of higher rank torus actions on $S$-arithmetic quotients of semi-simple or perfect algebraic groups 
	must be algebraic.
	\end{abstract}
	\maketitle
	
	\section{Introduction}\label{intro}
	
Higher rank diagonalizable actions have subtle rigidity properties, which are remarkable because the action of each individual element of the action has no rigidity. In this paper we establish one manifestation of this rigidity: a scarcity of joinings of fairly general higher rank actions on quotients of semisimple and more generally perfect linear algebraic groups, in stark contrast to the situation in the rank one case.
The systematic use of joinings as a powerful tool in ergodic theory was introduced by Furstenberg in his important paper \cite{Furstenberg-disjointness-1967}. The same paper also introduces the study of the rigidity properties of the action of a non-lacunary multiplicative semigroup of integers on $\R / \Z$, which is closely connected to the type of actions we consider here, and motivated much of the research in the area.

We recall that a joining of two measure preserving actions of a group $\Lambda$ on probability measure spaces $(X _ 1, \mathcal{B} _ 1, \mu _ 1)$ and $(X _ 2, \mathcal{B} _ 2, \mu _ 2)$ is a $\Lambda$-invariant probability measure $\rho$ on $X _ 1 \times X  _ 2$ (equipped with the product sigma-algebra $\mathcal{B} _ 1 \times \mathcal{B} _ 2$) whose projections to $X _ 1$ and $X _ 2$ are equal to $\mu _ 1$ and $\mu _ 2$ respectively. More generally one can (as we do in this paper) consider $r$-fold joinings of $(X _ i, \mathcal{B} _ i, \mu _ i)$ for $i = 1, \dots,r$ which would be a measure on $X _ 1 \times \dots \times X _ r$. While joinings can be considered for general group actions, in this paper we consider the case of $\Lambda = \Z ^ d$.
There is always at least one joining: the trivial joining given by $\rho = \mu _ 1 \times \dots \times \mu _ r$. Two $\Lambda$-actions for which the trivial joining is the only joining are said to be \emph{disjoint}. 

Joinings, especially scarcity of joinings, provide important information on the measure preserving $\Lambda$-action at hand. 
For instance, any nontrivial factor of the action of $\Lambda$ on a probability space $(X, \mathcal{B}, \mu)$, i.e. a measure-preserving map $\phi$ from $X$ to another probability space $(Y, \mathcal{C}, \nu)$ on which $\Lambda$ acts intertwining the action,
gives rise to nontrivial self-joining called the relatively independent joining $\mu \times _ Y \!\mu$ of $X$ with itself over~$Y$. The relatively independent self-joining has the property that~$\mu \times _ Y \!\mu$-a.e.~$(x_1,x_2) \in X \times X$ satisfies $\phi (x_1)= \phi (x_2)$, hence is certainly not the trivial joining. 

Scarcity of joinings is also useful for establishing equidistribution, allowing us under favorable conditions to upgrade equidistribution of a collection of points in two spaces to a joint equidistribution statement on an appropriate collection of points in a product space. 

The classification of joinings for higher rank abelian algebraic actions is a special case of the problem of classifying invariant measures under higher rank abelian groups. This measure classification question was raised by Furstenberg following his work on orbits of nonlacunary semigroups of integers on $\R / \Z$ in \cite{Furstenberg-disjointness-1967}, and conjectures in this vein in the context of actions on homogeneous spaces were made by Katok and Spatzier \cite[Main Conj.]{KatokSpatzier96} and in a more explicit form by Margulis \cite[Conj.\ 2]{Margulis-conjectures}; cf.\ also \cite[Conj. 2.4]{Einsiedler-Lindenstrauss-ICM}. Whether in the original context of actions of nonlacunary semigroups of integers or in the case of the action of higher rank abelian groups on homogeneous spaces all progress regarding this measure classification, e.g.~\cite{Rudolph-2-and-3,KatokSpatzier96,Lindenstrauss-03,Einsiedler-Katok-Lindenstrauss},
has been obtained under the assumption of positive entropy. Even assuming positive entropy, the state-of-the-art regarding classifications of invariant measures on quotients of semisimple groups requires that the acting group be a maximal split torus in at least some of the simple normal subgroups \cite{full-torus-paper}. By their very nature, for the class of actions we consider here joinings always have positive entropy; in fact, as we show in \S\ref{support of joinings}, not only do joinings have positive entropy, but the joining assumption also implies some finer information regarding the support of the corresponding leafwise measures on the coarse Lyapunov foliations (foliations whose leaves are obtained by intersecting unstable foliations of finitely many elements of the action) which is crucial for our argument. This allow us to classify joinings of essentially arbitrary higher rank actions on semisimple quotients.

In this paper we consider $S$-arithmetic quotients of $\Q$-perfect groups: more precisely, we take $\mathbb G$ to be a Zariski connected linear algebraic group over $\Q$, which we assume to be perfect, i.e.~$[\GG,\GG]=\GG$, and $S$ to be a finite set of places.
We set~$\Q_\infty=\R$ and write
$\Q_S=\prod_{\sigma\in S}\Q_\sigma$ for the corresponding product of the local
fields determined by the given set of places. An \emph{$S$-arithmetic quotient of $\GG$} is a quotient space of the form $\Gamma\backslash G$ where~$G\leq\GG(\Q_S)$ is a finite index subgroup, $\Gamma<G \cap \GG(\Q)$ is an arithmetic lattice commensurable to $\GG(\mathcal{O}_S)$,
and~$\mathcal{O}_S$ denotes the ring of $S$-integers in $\Q_S$. 

\begin{definition}\label{def:saturated}
Let $\Gamma\backslash G$ be an $S$-arithmetic quotient with $\GG, S, G, \Gamma$ as above.
We say that $\Gamma\backslash G$ is \emph{saturated by unipotents} if the group generated by all unipotent
elements of $\GG(\Q_S)$ acts ergodically on~$\Gamma\backslash G$.
\end{definition}

For $x \in \Gamma\backslash G$, it will be convenient for us to use the notation $g.x$ for $xg^{-1}$.
If $\GG$ is a $\Q$-almost simple group with $\GG(\Q_S)$ noncompact, and if $\Gamma\backslash G$ is an $S$-arithmetic quotient of $\GG$, then there is a finite index normal subgroup $G_1 \unlhd G$ so that $(\Gamma\cap G_1)\backslash G_1$ is saturated by unipotents. More generally, the same is true whenever $\GG$ is a $\Q$-group whose radical equals its unipotent radical $\urad$, and so that $\GG/\urad$ has no $\Q$-almost simple factor $\HH$ so that $\HH(\Q_S)$ is compact --- cf.\ Remark~\ref{saturated by unipotent remark}.

\begin{definition}\label{def:algebraic measure}
We say that a measure $\mu$ on an $S$-arithmetic quotient $\Gamma\backslash G$ as above is \emph{algebraic over $\Q$} if
there exists 
a~$\Q$-subgroup~$\HH\leq\GG$ and a finite index subgroup~$H\leq\HH(\Q_S)$
such that~$\mu$ is the normalized Haar measure on a single orbit~$\Gamma H g$
for some~$g\in G$.
\end{definition}

\begin{definition} \label{def:class-A}
A subgroup $A < G$ is said to be of \emph{class-$\cA'$} if it can be simultaneously diagonalized so that for every $a \in A$
\begin{itemize}
\item
 the projection of $a$ to $\GG(\R)$ has positive real eigenvalues if $\infty \in S $,
\item for each finite $p\in S$, the projection of $a $ to $\GG(\Q _ p)$ satisfies that all of its eigenvalues
are powers of $\lambda_{p}$, where $\lambda_{p}\in\Q_p^{*}$
is some fixed invertible element with $|\lambda_{p}|_{p}\neq 1$.
\end{itemize} 
An element $g \in G \setminus\{e\}$ is said to be of class-$\cA'$ if the group it generates is of class-$\cA'$. A homomorphism $a : \Z^d \to G$ is said to be of class-$\cA'$ if it is a proper map and its image $a(\Z^d)$ is of class-$\cA'$.
\end{definition}

Our definition of class-$\mathcal{A}'$ elements is more general than the notion of class-$ \mathcal{A}$ elements introduced by Margulis and Tomanov in \cite{Margulis-Tomanov}, though for the purposes of \cite{Margulis-Tomanov,Margulis-Tomanov-almost-linear} one can use this more general class equally well.

We can now state our first joining theorem:

\begin{theorem} 
\label{higher rank}
Let $r,d\geq 2$ and let $\GG_1,\ldots,\GG_r$ be
semisimple algebraic groups defined over~$\Q$ that are~$\Q$-almost simple, $\GG=\prod_{i=1}^r\GG_i$, and $S$ be a finite set of places of $\Q$.
Let $X_i = \Gamma _ i \backslash G _ i$ be $S$-arithmetic quotients saturated by unipotents for $G_i \leq \GG_i(\Q_S)$  and let $X= \prod_ i X_i$. Let $a_i : \Z ^ d \to G _ i$ be proper homomorphisms so that $a=(a_1, \dots, a_r) : \Z ^ d \to G= \prod_ i G_i$ is of class-$\mathcal{A}'$, and set $A = a(\Z ^ d)$. Suppose $\mu$ is an $A$-invariant and ergodic joining of the actions of $A_i = a_i(\Z^d)$ on $X_i$ equipped with the Haar measure $m_{X_i}$. Then $\mu$ is an {\em algebraic measure} defined over $\Q$.
\end{theorem}

We note that since it has to project onto each of the $\GG_i$, the $\Q$-group $\HH$ corresponding to $\mu$ as in Definition~\ref{def:algebraic measure} is semisimple (cf.\ Proposition~\ref{joining-groups} for details).

Theorem~\ref{higher rank} implies that joinings of such higher rank actions must be trivial unless there are strong relationships between at least some of the $\Q$-groups $\GG_1,\ldots,\GG_r$. Indeed, if the joining is non-trivial then two of the~$\Q$-subgroups~$\GG_i$ have the same adjoint form over $\Q$. We also note that 
in a sense made precise by the following corollary the case of two factors can be viewed as the most fundamental:

\begin{corollary}
\label{cor:pairs}
  Assume in addition to the assumptions in Theorem~\ref{higher rank} that for any two factors~$X_i$ and~$X_j$ with~$i\neq j$, the image of~$\mu$ under the natural projection map to~$X_i\times X_j$
  is the product measure~$m_{X_i}\times m_{X_j}$ .
Then~$\mu$ is the trivial joining, i.e.~$\mu=m_X=m_{X_1\times\cdots \times X_r}$.
\end{corollary}

We also consider the following more general setup: Let~$\LL=\GG_1\times\cdots\times\GG_r$ be a semisimple~$\Q$-group with $\Q$-almost
simple factors~$\GG_i$ for~$i=1,\ldots,r$, let~$\UU_{\textrm{rad}}$ be a 
Zariski connected unipotent
algebraic group defined over~$\Q$ such that~$\LL$
acts on~$\UU_{\textrm{rad}}$ by automorphisms and the representation is defined over~$\Q$.
We then define~$\GG=\LL\ltimes\UU_{\textrm{rad}}$, which is an algebraic
group defined over~$\Q$, with the canonical projection map~$\pi:\GG\to\LL$ whose kernel is~$\UU_{\textrm{rad}}$.
We will assume that~$\GG$ is perfect, i.e.~$\GG=[\GG,\GG]$
(or equivalently~$\UU_{\textrm{rad}}=[\LL,\UU_{\textrm{rad}}]$, since otherwise $[\GG,\GG]$ would be contained in the proper subgroup of $\GG$ generated by $\LL$, $[\LL,\UU_{\textrm{rad}}]$ and $[\UU_{\textrm{rad}},\UU_{\textrm{rad}}]$). 
We now set~$G=L\ltimes \UU_{\textrm{rad}}(\Q_S)\leq\GG(\Q_S)$, where~$L=G_1\times\cdots\times G_r$ and each $G_i$ a finite index sub-group of $\GG_i(\Q_S)$,  
and take~$\Gamma< G$ an arithmetic lattice 
compatible with the~$\Q$-structure of~$\GG$ and such that~$X=\Gamma\backslash G$ is saturated by unipotents. In order to minimize the difference in the setups, we will
assume for simplicity that the lattice $\pi(\Gamma)$ in $L$ is a product of lattices $\Gamma_1\times\cdots\times\Gamma_r$ with each $\Gamma_i$ a lattice in~$G_i$.

\begin{theorem}
\label{thm:unipotent}
	Let~$r\geq 1$, $d\geq 2$,~$G$ and $X$ be as above. Let $a_i : \Z ^ d \to G _ i$ be proper homomorphisms so that $a=(a_1, \dots, a_r) : \Z ^ d \to G= \prod_ i G_i$ is of class-$\mathcal{A}'$, and set $A = a(\Z ^ d)$. Let~$\mu$ be an~$A$-invariant and ergodic probability measure
	on~$X$ such that~$\pi_*(\mu)$ is the Haar measure on~$(\Gamma_1\times\cdots\times\Gamma_r)\backslash L$.
	Then~$\mu$ is an algebraic measure defined over~$\Q$.
\end{theorem}

Combining Theorem~\ref{higher rank} and Theorem~\ref{thm:unipotent} we can get a classification of joinings of class-$\mathcal{A}'$ actions on quotient spaces of perfect groups analogous to Theorem~\ref{higher rank}:

\begin{theorem}\label{thm:perfect}
Let $r,d\geq 2$ and let $\GG_1,\ldots,\GG_r$ be perfect algebraic groups defined over~$\Q$, $\GG=\prod_{i=1}^r\GG_i$, and $S$ be a finite set of places of $\Q$.
Let $X_i = \Gamma _ i \backslash G _ i$ be $S$-arithmetic quotients for $G_i \leq \GG_i(\Q_S)$ saturated by unipotents and let $X= \prod_ i X_i$. Let $a_i : \Z ^ d \to G _ i$ be homomorphisms so that $a=(a_1, \dots, a_r) : \Z ^ d \to G= \prod_ i G_i$ is of class-$\mathcal{A}'$, and such that the projection of $a_i$ to every $\Q$ almost simple factor of $\GG_i(\Q_S)$ is proper.
Suppose $\mu$ is an $A$-invariant and ergodic joining of the action of $A_i = a_i(\Z^d)$ on $X_i$ equipped with the Haar measure $m_{X_i}$. Then $\mu$ is an algebraic measure defined over $\Q$.
\end{theorem}

Our assumption that the acting group $A$ is of class-$\mathcal A'$ is simultaneously necessary for the theorems above to hold as stated, and not particularly material since these theorems can be modified in a fairly straightforward way to include much more general diagonalizable groups. We provide extensive discussion of this issue in \S\ref{construction}.

\subsection{Arithmetic applications}
There seems to be substantial scope for applying the theorems presented above to arithmetic questions. A concrete result in this direction by M.~Aka, U.~Shapira and the first named author \cite{AES} considers for all integer points $(x,y,z)$ of the sphere $x ^2 + y ^2 + z ^2 = m$ both the normalized point $m ^ {-1/2} (x,y,z)$ on the unit sphere $S^2$ and the shape of the lattice orthogonal to $(x,y,z)$ in $\Z ^ 3$, which corresponds to a point in $\PGL (2, \Z) \backslash \H$. Using Theorem~\ref{higher rank} it is shown that these collections of \emph{pairs} in $S ^2 \times \PGL (2, \Z) \backslash \H$ become equidistributed as $m\to \infty$ along
squarefree\footnote{As explained in \cite{AES}, the square free assumption is in fact not needed, though without it the required number theoretical facts are not well documented.} integers, satisfying both $m \not\equiv 0,4,7 \bmod 8$ (a necessary and sufficient condition for existence of primitive integral points on the sphere of radius $\sqrt{m}$) as well as an auxiliary condition of Linnik type---namely that~$(-m)$
is a quadratic residue modulo two distinct fixed primes. In an appendix of that paper by Ruixiang Zhang a connection is made between this equidistribution result and the theory of Eisenstein series for maximal parabolic subgroups of $\SL(3)$.

This result can be viewed as a special case of the following more general arithmetic  equidistribution result (though to us the above application came as a surprise).
In order to be able to state concisely cases in which the assumptions of the theorem below can be verified we state it in terms of the adeles, though strictly speaking for the purposes of this paper this is mostly a matter of style rather than substance.
For further discussion of why and how the adelic formulation is natural in this context we refer the reader to \cite{ELMV3}.
If $\GG$ is a $\Q$-group and $\mu$ is a measure on a space on which $\GG(\A)$ acts we say that $\mu$ is \emph{nearly invariant} under $\GG(\A)$ if it is invariant under a finite index subgroup of $\GG(\Q_s)$ for any $s \in \{\infty,\textup{primes}\}$. 
Note that if $\GG$ is semisimple and simply connected, by strong approximation any probability measure on $\GG(\Q)\backslash \GG(\A)$ nearly invariant under~$\GG(\A)$ is in fact invariant, hence is equal to the uniform measure on $\GG(\Q)\backslash \GG(\A)$.

\begin{theorem}\label{arithmetic theorem}
Let $\A$ be the ring of adeles over~$\Q$, and let $\mathbb G _ 1$, \dots, $\mathbb G _ r$ be $\Q$-almost simple groups 
such that no two of these groups are isogenous over~$\Q$. Let $S _ 0$ be a finite set of places for $\Q$, and $N\in \N$. Let $\mathbb T_\alpha$ be a sequence of $\Q$-anisotopic $\Q$-tori. We think of both the $\mathbb G _ k$ and the $\mathbb T_\alpha$ as algebraic groups embedded in $\SL(N)$ for some $N$, and require that they are defined by polynomials of degree $\leq N$.
For each $\alpha$ 
let $m_\alpha$ be a probability measure on $\mathbb T_\alpha(\Q)\backslash \mathbb T_\alpha(\A)$ (e.g.\ the Haar measure).
For every $\alpha$ and $k=1, \dots,r$ let $\iota ^{(\alpha)} _ k : \mathbb T_\alpha \to \mathbb G _ k$ be homomorphism, given by polynomials of degree $\leq N$ and with zero dimensional kernel, 
and $g _ k ^ {( \alpha )}$ an arbitrary element in $\mathbb G _ k (\A)$. 
Assume the following two conditions hold:
\begin{description}[font=\normalfont,leftmargin=2\parindent,labelindent=\parindent]
\item[Uniform split rank] For any $\alpha$, the group $\mathbb T_\alpha (\Q _ {S _ 0})$ has $\Q_{S_0}$-rank $\geq 2$ and $m_\alpha$ is invariant under a subgroup of the $\Q_{S_0}$-split part of $\mathbb T_\alpha (\Q _ {S _ 0})$ of index $\leq N$.
\item[Individual equidistribution] For any $k$, the measures $g_k ^ {( \alpha )}. \bigl (\iota _ k ^ {( \alpha )} \bigr)\! \vphantom{|} _ {*} (m _ \alpha)$ on $X^{(k)}_{\A}=\mathbb G _ k (\Q)\backslash \mathbb G _ k (\A) $ converge as $\alpha\to\infty$ to a nearly $\GG_k(\A)$ invariant probability measure.
\end{description}
Then any limit point of the sequence of measures on
\(
 X_{\A}= X_{\A}^{(1)}\times\dots\times X_{\A}^{(r)}
\)
obtained by pushing $m _ \alpha$ forward via the map
\begin{equation*}
t \mapsto (g^{(\alpha)}_1. \iota ^ {( \alpha )} _ 1(t), \dots, g^{(\alpha)}_r. \iota ^ {( \alpha )} _ r(t))
\end{equation*}
is nearly invariant under $\GG_1\times\dots\times\GG_r(\A)$.
\end{theorem}

Theorem~\ref{arithmetic theorem} is proved in \S\ref{arithmetic section}.
In addition to Theorem~\ref{higher rank}, following Eskin, Mozes and Shah \cite{Eskin-Mozes-Shah,Eskin-Mozes-Shah-nondivergence}, 
Ratner's measure classification theorem \cite{Ratner-Annals} (or more accurately its generalization to the $S$-arithmetic context by Ratner \cite{Ratner-padic} and Margulis-Tomanov \cite{Margulis-Tomanov}) is used to handle the cases where the torus is distorted enough by the $g_k^{(\alpha)}$ to become unipotent in the limit.

A key point is that the individual equidistribution condition is known to hold in several natural and very interesting classes of $\Q$-groups $\GG$. Since individual equidistribution concerns only one group, we may as well assume the sequence of anisotropic $\Q$-tori $\mathbb T_\alpha$ are already embedded in $\GG$. 
To be able to state the known results we will need to have some way of measuring complexity (or volume) of a shifted adelic torus orbit 
$\GG(\Q) \backslash \GG(\Q) \mathbb T(\A) g^{-1}$, and we shall use the one from \cite{ELMV3}:
\[
\vol(\GG(\Q) \backslash \GG(\Q) \mathbb T(\A) g^{-1}) = \vol\{t \in T(\A): g t g^{-1} \in \Omega_0\}^{-1}
\]
where $\Omega_0$ is a fixed compact neighborhood of the identity in $\G(\A)$ and the volume on $\mathbb T(\A)$ is normalized so that $\mathbb T(\Q) \backslash \mathbb T (\A)$ has volume 1.

The situation is most satisfactory when $\GG=\PGL(2,F)$, for a number field $F$  (with $\PGL(2,F)$ considered as a $\Q$-group by restriction of scalars) or $\PGL(1, B_F)$ where $B_F$ is a quaternion division algebra over a number field $F$, which we again consider as a $\Q$-group (here we use $\PGL(1, B_F)$ to denote the groups of invertible elements of the quaternion algebra modulo the center of $B_F$). 
A combination of the results of several important papers (for instance, Duke's paper \cite{Duke-hyperbolic} which gives a result close to the assertion below for $F=\Q$) can be summarized by saying that in this case if $\mathbb T_\alpha$ is a sequence of maximal (in the sense of dimension) anisotropic tori in $\GG$, \ $g_\alpha \in \GG(\A)$ with $\vol(\GG(\Q) \backslash \GG(\Q) \mathbb T_\alpha(\A) g_\alpha^{-1})\to\infty$ and if $m_\alpha$ is the Haar measure on $\mathbb T_\alpha (\Q) \backslash \mathbb T_\alpha(\A)$ (which we consider as a probability measure on $X_\A=\GG(\Q)\backslash \GG(\A)$) then any weak$^*$-limit of $g_\alpha.m_\alpha$ as $\alpha \to \infty$ is nearly invariant under~$\GG(\A)$, and moreover appropriately formulated this convergence holds with a polynomial rate\footnote{In particular, for any smooth test function $f$ that depends on only a finite set of places~$S$, there is a function $\bar f$ invariant under a finite index subgroup of $\GG(\Q_S)$, so that $\absolute{\int (f - \bar f) \,d (g_\alpha.\mu_\alpha)} \ll_f\vol \bigl(\GG(\Q) \backslash \GG(\Q) \mathbb T_\alpha(\A) g_\alpha^{-1}\bigr)^{-\delta}$ for some $\delta>0$. The subgroup depends only on $S$, not on $f$.}.
By the work of Harcos and Michel \cite{Harcos-Michel} it is even possible to consider in this case measures $m_\alpha$ that are given by Haar measures on large enough subgroups of $\mathbb T_\alpha (\Q) \backslash \mathbb T_\alpha(\A)$. For a complete list of references and further discussion see \cite[\S4.5--4.6]{ELMV3}.
Under the splitting condition we make here, and without rates, this result can also be deduced (at least for $F=\Q$) from the work of Linnik and Skubenko \cite{Linnik-book}. 

Another case where the individual equidistribution condition has been established (without rate) is for $\GG=\PGL(3)$ and $\mathbb T_\alpha$ a sequence of maximal $\Q$-tori in $\GG$ with a splitting condition that is stronger than that assumed here, namely that for some place $s \in \{\infty,\textup{primes}\}$ the tori $\mathbb T_\alpha$ are $\Q_s$-split. In this case Ph. Michel, A. Venkatesh and the authors \cite{ELMV3} show similarly that if $g_\alpha$ is a sequence of elements in $\GG(\A)$ so that $\vol(\GG(\Q) \backslash \GG(\Q) \mathbb T_\alpha(\A) g_\alpha^{-1})\to\infty$  then any
weak~$^*$ limit of $g_\alpha.m_\alpha$ 
is nearly invariant under~$\GG(\A)$.

\subsection*{Acknowledgements} 
We thank the anonymous referee for her (or his) careful reading of our paper and comments. We also thank Menny Aka, Yves Benoist, Phillipe Michel, Uri Shapira, Jacob Tsimerman and Akshay Venkatesh for helpful discussions.

\section{Notation and preliminaries}

\subsection{Linear algebraic groups}\label{sec-places}

We recall some basic properties of algebraic groups, and refer the reader to \cite{Springer}. We use this theory as a natural framework that makes no distinction between the real and $p$-adic numbers. 

Let $ \Q _ \infty = \R$ and let
$ \sigma$ be $ \infty$ or a prime $p$ so that $\Q _ \sigma$ is either $\R$ or $\Q_p$. Let $| \cdot|_\sigma$ denote the absolute value if $ \sigma = \infty$ or the $p$-adic norm if $\sigma = p$.

Let $G=\G(\Q _ \sigma)$  be the $\Q _ \sigma$-points of a Zariski connected linear 
algebraic group $\G$ defined over $ \Q _ \sigma$. 
A subgroup $U < G$ is {\em unipotent} if for every $g \in U$,
$g-e$ is a nilpotent matrix, i.e.\ for some $n$, \ $(g-e)^n=0$ where $e$ is
identity. A subgroup $H\leq G$ is said to be {\em normalized by} $g  \in G$ if $ g H g ^{-1} =
H$; $H$ is normalized by $L\leq G$ if it is normalized by every $g  \in
L$; and the {\em normalizer} $N_G(H)$ of $H$ is the group of all $g
\in G$ normalizing it.
Furthermore, $N_G^1(H)$ denotes the subgroup of the normalizer that in addition also preserves the Haar measure on $H$.
 Similarly, $g$ {\em centralizes} $H$ if
$gh=hg$ for every $h  \in H$, and we set $C_G(H)$, the {\em
centralizer} of $H$ in $G$, to be the group of all $g  \in G$
centralizing $H$.

\subsection{Weights and Lyapunov weights}
Let $ \mathfrak g $ be the Lie algebra of $ G $ and let
\begin{equation*}
\operatorname{Ad} _ g: \mathfrak g \rightarrow \mathfrak g \qquad\mbox{ for } g \in G
\end{equation*}
be the adjoint representation of $ G $ on $ \mathfrak g $.

Let $a$ be a class-$\mathcal{A}'$ homomorphism $\Z ^ d \to G$ and $A = a (\Z ^ d)$.
Then the adjoint action 
of $A$ on the Lie algebra $\mathfrak g$ of $G$ decomposes $\mathfrak g$ into a 
direct sum of joint eigenspaces of all elements of $A$. Each such joint 
eigenspace corresponds to a character $\alpha:A\to\Q_\sigma^\times$ which we 
consider also as a map from $\Zd$ to $\Q_\sigma^\times$ (using the homomorphism~$a:\Zd\to A$). It follows from Definition~\ref{def:class-A} that $\alpha(\n)=\lambda_p^{\n\cdot\ba_\alpha}$ for some fixed
$\ba_\alpha\in\Zd$ and $\lambda_p\in\Q_p$ if~$\sigma=p$ and~$\alpha(\n)=e^{\n\cdot\ba_\alpha}$
for some~$\ba_\alpha\in\R^d$ if~$\sigma=\infty$. We will refer to 
$\alpha$ as a {\em weight} and to the eigenspace $\mathfrak g^\alpha$ as the 
{\em weight space}. 
The eigenspace for the {\em trivial weight} will 
be denoted by $\mathfrak g^0$, i.e.\ $\mathfrak g^0$ is the Lie algebra of 
$C_G(A)$, the centralizer of $A$ in $G$.

For each weight $\alpha$ we define its \emph{Lyapunov weight} $\log|\alpha(\n)|_\sigma=(\n\cdot\ba_\alpha)\log|\lambda|_\sigma$ which is a linear map from~$\Z^d$ to $\R$. The Lyapunov weight $0$ corresponding to the trivial weight we again call trivial. Below we will always work with Lyapunov weight since they are better suited for a uniform treatment of all places.  The class-$\mathcal A'$
assumption implies in the case of a finite place~$\sigma=p$
that the corresponding Lyapunov weights are real multiples of rational linear maps. 

Let $\Phi$ denote the set of nontrivial Lyapunov weights  $\alpha$ (with non-trivial weight space). For the dynamical properties of the action another notion is of fundamental importance. We say that two Lyapunov weights $\alpha,\beta\in \Phi$
are \emph{equivalent} if there is some $c>0$ with $\alpha=c\beta$.  The equivalence classes are called \emph{coarse Lyapunov weights} and for a given coarse Lyapunov weight $[\alpha]$ the sum of the weight spaces $\mathfrak g^{[\alpha]}=\sum_{\beta\in[\alpha]}\mathfrak g^\beta$ is the \emph{coarse Lyapunov weight space}. We denote the set of all coarse Lyapunov weights by $[\Phi]$.

\subsection{The Lie algebra and the exponential map}

In this section we recall standard facts and notations, which as phrased below also work over
non-archimedean places (see e.g.~\cite[\S 3]{EinsiedlerKatokNonsplit}). Let
\begin{equation*}
[\cdot, \cdot]: \mathfrak g ^ 2 \rightarrow \mathfrak g
\end{equation*}
be the $\Q _ \sigma$-bilinear Lie bracket satisfying
\begin{equation} \label{Adjoint is homomorphism}
\operatorname{Ad} _ g ([u, v]) = [\operatorname{Ad} _ g u, \operatorname{Ad} _ g v]
\end{equation}
for all $ g \in G $ and $ u, v \in \mathfrak g $.
Since $\Q _ \sigma$ is a local field of characteristic zero, the exponential map $ \exp (\cdot )$ is a local homeomorphism between the
Lie algebra $ \mathfrak g $ and $ G $
such that
\begin{equation} \label{Conjugation and adjoint}
g \exp (u ) g ^ {- 1} = \exp (\operatorname{Ad} _ g (u))
\end{equation}
whenever both sides are defined.  Write $ \log (\cdot) $ for the locally defined inverse map.
Furthermore, we define $\operatorname{ad}_u(v)=[u,v]$ for any $u, v \in \mathfrak g$. Then
\begin{equation} \label{Adjoint and adjoint}
\operatorname{Ad} _{\exp u} = \exp
(\operatorname{ad}_u)
\end{equation}
whenever both sides are defined.

We have
\begin{equation} \label{weights add}
[\mathfrak g ^ \alpha, \mathfrak g ^ \beta] \subset \mathfrak g ^ {\alpha + \beta}
\end{equation}
for any Lyapunov weights $\alpha, \beta$ (which follows easily from \equ{Adjoint is 
homomorphism}).
If $ \alpha \in \Phi$ is a Lyapunov weight, then the above implies that the coarse 
Lyapunov weight space $\mathfrak g^{[\alpha]}$ is a Lie algebra. Moreover, the 
exponential map  restricted to this subalgebra is a polynomial map and 
so can be extended to the whole of $ \mathfrak g ^ {[\alpha]}$ with the image $ G ^ {[\alpha]} = \exp 
\mathfrak g ^ {[\alpha]}$ a unipotent group. This extension is the unique way to extend $\exp$ to $\mathfrak g^{[\alpha]}$ so that that \equ{Conjugation and adjoint} holds. More generally, let $ \Psi \subset \Phi$ be a set of Lyapunov weights  such that
$ (\Psi + \Psi) \cap \Phi \subset \Psi$ and $ \alpha (a) > 0$
for all $ \alpha \in \Psi$ and some fixed $ a \in A$.  Then
\begin{equation*}
\mathfrak g ^ \Psi = \sum_{\alpha \in \Psi} \mathfrak g ^ {\alpha}
\end{equation*}
is a Lie algebra, $ \exp ( \cdot )$ can be uniquely extended to $ \mathfrak g ^ \Psi$ such that \equ{Conjugation and adjoint} holds, and $ G ^ \Psi = \exp \mathfrak g ^ \Psi$
is a unipotent subgroup that is normalized by~$A$ and generated by the subgroups $G ^ {[\alpha]}\cap G^\Psi$ for $\alpha \in \Psi$. It will be convenient to say that a subgroup $U ' \subset G ^ \Psi$ is Zariski connected if there exists a Lie subalgebra $\mathfrak g ' \subset \mathfrak g ^ \Psi$ such that $U' = \exp (\mathfrak g ')$; in the real case this notion agrees with the subgroup being connected with respect to the Hausdorff topology.

If $\Psi=\{\alpha\in\Phi:\alpha(a)<0\}$ for some $a\in A$, then the Lie algebra $\mathfrak g_a^-=\mathfrak g^\Psi$ is called the \emph{stable horospherical Lie algebra} and the associated unipotent subgroup $G_a^-$ is the \emph{stable horospherical subgroup}. Similarly $\mathfrak g_a^+=\mathfrak g_{a^{-1}}^-$ and $G_a^+=G_{a^{-1}}^-$ are the \emph{unstable horospherical Lie algebra and subgroup} defined by $a$.

\section{Necessity of the assumptions on the acting group}\label{construction}

\subsection{A counterexample}\label{counter-section}
To see how the conclusion of Theorem~\ref{higher rank} may fail if the homomorphism $a : \Z^d\to\prod_iG_i$ is not of class-$\cA'$, we give a simple construction of a counterexample. As we will see the measure in this counterexample still has a lot of algebraic structure but fails the precise description in Theorem~\ref{higher rank}.

Let~$X_1=\Gamma_1\backslash G_1$ be an arithmetic quotient of a semisimple group that is saturated by unipotents as in Definition~\ref{def:saturated}.
Let~$a_1 : \Z^d \to G_1$ be a class-$\cA'$ homomorphism, $A_1 = a_1(\Z^d)$. Suppose in addition that the centralizer of~$A_1$ contains an infinite
compact subgroup, and choose some~$k_1,\ldots,k_d \in G_1$ which are~$d$ commuting elements
in the centralizer of~$A_1$ which generate a dense subgroup of an infinite compact group~$M$.  
For simplicity we also assume that~$\Gamma_1$ is torsion free so 
that~$M\cap g^{-1}\Gamma_1g=\{e\}$ for all~$g\in G_1$. 
We define~$A_2<G_2=G_1$ by setting~$a_2(\mathbf{n})=a_1(\mathbf{n})k_1^{n_1}\cdots k_d^{n_d}$
for all~$\mathbf{n}\in\ZZ^d$. 
We also set~$\Gamma_2=\Gamma_1$ so that~$X_2=X_1$ and~$X=X_1\times X_1$. 
As before we let~$a=(a_1,a_2)$ and take $A=a(\Z^d)$. We note that $a$ is not of class-$\cA'$, even though one can easily construct~$p$-adic
examples where both~$a_1$ and~$a_2$ are separately of class-$\cA'$.

We now wish to define an~$A$-invariant and ergodic joining on~$X$ that is not algebraic in the sense of Definition~\ref{def:algebraic measure}. Let~$\mu_1$
be the Haar measure on the diagonal~$\{(x_1,x_1):x_1\in X_1\}$. Clearly~$\mu_1$
projects under the natural projections $\pi_i:X\to X_i$ ($i=1,2$) to the Haar measures on~$X_1$ and~$X_2$, but it is not invariant under the diagonally embedded subgroup~$A=\{(a_1(n),a_2(n)):n\in\ZZ^d\}$. However, let $\mu$ denote the probability measure
\[
 \mu=\int_M (e,m). \mu_1\operatorname{d}\!m_M(m),
\]
where the integral is taken with respect to a probability Haar measure~$m_M$ on~$M$, and $g.\mu_1$ denotes the pushforward of $\mu_1$ under the natural action of $G=G_1\times G_2$ from the right on $\Gamma\backslash G$. Then $\mu$ 
is an~$A$-invariant joining for the action of~$A_1$ on~$X_1$ and
the action of~$A_2$ on~$X_2$.  This measure is not algebraic (as defined
in Theorem~\ref{higher rank}) as~$H=\bigcup_{g\in G_1}(\{g\}\times Mg)$ 
is not a subgroup of~$G$ (as the latter would imply
that~$M\cong H\cap(\{e\}\times G_1)$ is a normal subgroup of~$G_1$). Since~$M\cap g^{-1}\Gamma_1g=\{e\}$
for all~$g\in G_1$
we see that~$(X,\mu)$ is isomorphic to~$(X_1\times M,m_{X_1}\times m_M)$, and that the action of~$A$ is taken under this isomorphism to the action of~$\{(a_1(\mathbf{n}),k_1^{n_1}\cdots k_d^{n_d}):\mathbf{n}\in\ZZ^d\}$. By a variant of Moore's ergodicity theorem (cf.~\cite[\S II.3]{Margulis-book}), the action of~$A_1$ on~$X_1$ is weakly mixing (since we also use this fact later, we state and prove it below in Proposition~\ref{fancy-Mautner}). As the elements~$k_1,\ldots,k_d$ generate 
a dense subgroup of~$M$, it follows that the above action on $X_1\times M$ is ergodic, and hence~$\mu$ is indeed an ergodic non-algebraic joining.

\begin{proposition}\label{fancy-Mautner}
	Let~$\GG=\LL\ltimes\UU_{\textrm{rad}}$, $\LL=\GG_1\times\cdots\times\GG_r$, and~$\UU_{\textrm{rad}}$ 
	be algebraic groups defined over~$\Q$ as in Theorem~\ref{thm:unipotent}; in particular~$\GG$ is
	a perfect group. Let~$S$ be a finite set
	of places of~$\Q$, $G\leq \GG(\Q_S)$ be a 
	finite index subgroup, and $\Gamma<G \cap \GG(\Q_S)$ an arithmetic
	lattice so that~$X=\Gamma\backslash G$ is saturated by unipotents. Let~$a\in G$ be an element of class-$\cA'$ such that its projection to each~$\Q$-almost simple factor $\GG_i$ of~$\LL$ is nontrivial. Then the action of~$a$
is weak mixing w.r.t.\ $m_X$.
\end{proposition}

\begin{proof}
	Fix some~$f\in L^2(X,m_X)$, which is an eigenfunction w.r.t.~$a$.
	Let~$u\in G^-_a$ be such that~$a^n ua^{-n}$ converges to the identity as~$n\to\infty$. Then~$f$ is
	invariant under~$u$. We recall the simple argument (c.f.~\cite[Lemma II.3.2]{Margulis-book})
	\[
	  \| f(x u)-f(x)\|_2=\|f(x u a^{-n})-f(x a^{-n})\|_2=
	\|f(x a^nua^{-n})-f(x)\|_2,
	\]
	where we used that~$f$ is an eigenfunction w.r.t.~$a$ (and that the eigenvalue
	must have absolute value one) in the first step and unitarity
	of the action of~$a$ in the second. However, the last expression now converges
	to~$0$ as~$n\to\infty$. This implies that~$f$ is invariant from the right under~$G^-_a$ and similarly under~$G_a^+$.

	We let~$H$ be the subgroup of~$G$ generated by~$G_a^-$ and~$G_a^+$ and recall that the Lie algebra of~$H$ is a Lie ideal
	of~$G$ (the Auslander ideal).
	It follows that a finite index subgroup of $H$ is normal in $G$, but since both $G_a^+$ and $G_a^-$ have no finite index subgroups it follows that in fact $H$ itself is normal.
	
	\label{plus paragraph}
	Recall that for every isotropic~$\Q_\sigma$-simple algebraic group~$\mathbf F$ 
	there exists a minimal finite index subgroup of~$F=\mathbf F(\Q_\sigma)$, namely the image~$F^+=[F,F]$
	of the group of~$\Q_\sigma$-points of the simply connected cover~$\tilde{\mathbf F}$ of~$\mathbf F$,
	see~\cite[I.2.3]{Margulis-book} and the references therein. Moreover,~$F/F^+$
	is abelian. 
	It follows from the assumptions on $a$ that for every $\Q$-almost simple factor $\GG_i$ of $\LL$, there is a $\sigma \in S$ so that $G_a^- \cap \GG_i(\Q_\sigma)$ is infinite, hence by normality of $H$ and minimality of $F^+$ we must have that $H$ contains $F^{+}$ 
	for some $\Q_\sigma$-simple factor $\mathbf F$ of $\GG_i$ (considered as an algebraic group over $\Q_\sigma$).
	
	We may identify $f$ with a left $\Gamma$-invariant function on $G$, and using this identification, we
	define~$I_f=\{g\in G: f$ is left~$g$-invariant$\}$. As $H$ is normal, $f$ is~$H$-left and right invariant. 
	Let  $\UU_0 = \urad/[\urad,\urad]$ and \[\Gamma_0 = (\Gamma \cap \urad(\Q_S)) \bmod [\urad(\Q_S),\urad(\Q_S)].\] Note that $\LL\ltimes\UU_0$ inherits from $\LL\ltimes\urad$ a natural structure of an algebraic group defined over $\Q$. 
	
	We claim the image of $I_f \cap \urad(\Q_S)$ under the natural quotient map \[\iota:\urad(\Q_S)\to\UU_0(\Q_S)\]
	is all of $\UU_0(\Q_S)$, from which it follows by a standard induction argument on the nilpotency degree that $I_f > \urad(\Q_S)$.
	Suppose then that \[\iota(I_f \cap \urad(\Q_S)) \subsetneqq \UU_0(\Q_S).\] Then as $I_f$ is closed and contains $\Gamma$, the image $\iota(I_f \cap \urad(\Q_S))$ is a (proper) closed subgroup of the solenoid $\UU_0(\Q_S)/\Gamma_0$. 
	
	Since $H$ was normal in $\LL \ltimes\urad (\Q_S)$, its image $H_0$ in $\LL \ltimes\UU_0 (\Q_S)$ is also normal.
	It follows that 
\begin{equation}\label{H UU_0 equation}
	[H_0,\UU_0(\Q_S)] \subset \UU_0(\Q_S) \cap H_0.
\end{equation}
	As
\[
\Gamma _ 0 (\UU _ 0 (\Q _ S) \cap H _ 0) \subseteq \iota (I _ f \cap \urad (\Q _ S))
\]
there is a proper algebraic subgroup $\VV < \UU _ 0$ defined over $\Q$ so that
\begin{equation}\label{H UU_0 equation cont.}
\UU _ 0 (\Q _ S) \cap H _ 0 \subset \VV (\Q _ S)
.\end{equation}

	As we have noted above, for any $i$ the group $\GG_i$ has a $\sigma \in S$ and a $\Q_\sigma$-simple group $\mathbf F$ so that $F^+ \unlhd \GG_i(\Q_\sigma)$ and $F^{+} \leq  H$ (or, if we consider $F^+$ as a subgroup of $\LL\ltimes\UU_0(\Q_S)$, \ $F^+\leq H_0$). By \eqref{H UU_0 equation} and \eqref{H UU_0 equation cont.} we have that $[F^{+},\UU_0(\Q_\sigma)]\leq  \VV(\Q_\sigma)$ which implies $[\GG_i,\UU_0] \leq  \VV$. But this implies $[\LL, \UU_0] \leq  \VV$ which is a proper subgroup of $\UU_0$, contradicting $[\LL,\UU]=\UU$.  This finishes the proof of the claim and hence~$f$ is invariant under~$\urad(\Q_S)$.

	For notational convenience, we assume now (as we may by the above paragraph) that~$\UU_{\textrm{rad}}$ is trivial.
	Using our assumption on~$a$ and 
	strong approximation on the simply connected cover 
	of~$\LL$ (see e.g.~\cite[\S II.6.8]{Margulis-book}) it follows that~$I_f>\Gamma H$ has finite index in~$\LL(\Q_S)$. 
	Hence there exists a finite index normal subgroup~$I'\unlhd G$ contained in~$I_f$, so that~$f$
	is also right invariant under~$I'$. However, as a unipotent subgroup does not
	have any finite index subgroup, it follows that~$I'$ contains all unipotent elements
	of~$G$. By the assumption that $\Gamma\backslash G$ is saturated by unipotents it follows that~$I'$ acts ergodically, which implies that~$f$ is constant.	
\end{proof}

\begin{remark}\label{saturated by unipotent remark}
A similar (simpler) arguments can be used to establish the claim following Definition~\ref{def:saturated}. Indeed,
suppose $\GG$ is an algebraic group defined over $\Q$ whose radical equal to its unipotent radical $\urad$, so that $\GG / \urad$ has no $\Q$-almost simple factor $\GG_1$ with $\GG_1 (\Q_ S)$ compact.
Let $H$ be a subgroup of $\GG (\Q _ S)$ generated by unipotent subgroups. By definition, $H \geq  \urad (\Q _ S)$. Applying strong approximation on each $\Q$-almost simple factor of $\GG / \urad$ as above, we may conclude that
\begin{equation*}
\overline { \Gamma H} = G
\end{equation*}
for some finite index subgroup $G < \GG (\Q _ S)$, which implies the claim.
\end{remark}

\subsection{General extension of Theorem~\ref{higher rank}}

We now explicate how to generalize the classification of joinings from the case
considered in Theorem~\ref{higher rank} to the more general class of
higher rank actions without any eigenvalue assumption.

Suppose
that~$X=X_1\times\cdots\times X_r$ for~$r\geq 2$ and that~$X_i=\Gamma_i\backslash G_i$ 
for~$i=1,\ldots,r$ are as in Theorem~\ref{higher rank}.
Let $a_i : \Z^d \to G_i$  for~$i=1,\ldots,r$ be proper homomorphism (i.e. group homomorphisms that are also proper as maps from one topological space to another) with diagonalizable image (i.e. for every~$\sigma\in S$ the projection of~$a_i$ to the~$\sigma$-adic 
factor of~$G_i$ is diagonalizable over the algebraic closures 
of~$\Q_\sigma$). Let~$A_i = \{a_i(\mathbf n):\mathbf n\in\ZZ^d\}$ and $A=\{(a_1(\mathbf n),\ldots,a_r(\mathbf n)):\mathbf n\in\ZZ^d\}$ 
be the diagonally embedded subgroup. 

\begin{corollary}\label{cor-extension}
 Let~$X=X_1\times\cdots\times X_r$ 
 and~$A<A_1\times\cdots\times A_r<G=G_1\times\cdots\times G_r$ be
 as above for some~$r\geq 2$ (in particular, each $G_i$ is a finite index subgroup of the $\Q_S$-points of a $\Q$-almost simple group). Let~$\mu$ be an~$A$-invariant and ergodic joining on~$X$.
Then there exists a closed abelian subgroup~$A'<G$ containing~$A$ such that~$A'/A$
is compact, an algebraic
semi-simple subgroup~$\mathbb H<\mathbb G$ defined 
over~$\Q$, and a finite index subgroup~$H<\mathbb H(\Q_S)$ such that~$\mu$ is the measure
\[
 \mu=\int_{A'/A_H} a'. m_{\Gamma  H g} \operatorname{d}\!m_{A'/A_H}(a'A_H),
\]
for some~$g\in G$, where~$A_H=A'\cap g^{-1}H g$ is a cocompact subgroup of~$A'$ and~$m_{A'/A_H}$
denotes the normalized Haar measure on the compact group~$A'/A_H$.
\end{corollary}

\begin{proof}
	The proof uses the decomposition of semi-simple elements into compact and non-compact 
	semi-simple elements together with the argument from 
	\S\ref{counter-section}. 
	
	Recall that in any algebraic group~$\GG$ defined over~$\R$ any semi-simiple  
	element~$a\in\GG(\R)$
	can be split into a commuting product~$a=k_{a}a_{\textrm{nc}}$ 
	of a compact semi-simple
  element~$k_{a}\in\GG(\R)$ and a non-compact element~$a_{\textrm{nc}}\in\GG(\R)$
  such that the eigenvalues of~$k_{a}$ are all of absolute value one
  and the eigenvalues of~$a_{\textrm{nc}}$ are all real and positive. This decomposition
	is unique. For several commuting elements the resulting elements still commute with each 
	other and furthermore, the non-compact elements all belong to the 
	connected component~$T^\circ$ of the~$\R$-points~$T=\mathbb{T}(\R)$ of an~$\R$-split 
	subtorus $\mathbb{T}$ of~$\GG$.
	If~$G_\R<\GG(\R)$ 
	is a finite index subgroup and~$a\in G_\R$, then~$a_{\textrm{nc}}\in T^\circ<G_\R$ 
	implies that~$k_a,a_{\textrm{nc}}\in G_\R$.
	
	With some modifications, a similar splitting can be concocted in an algebraic group~$\GG$ defined over~$\Q_p$.
	We may extend the~$p$-adic absolute value from~$\Q_p$ to its algebraic closure.
	Let~$a\in\GG(\Q_p)$ be a semi-simple element. If the eigenvalues
	of~$a$ all have~$p$-adic absolute values which are integer powers of~$p$, then
	we can define a splitting analogous to the real case: there exists 
	a unique~$a_{\textrm{nc}}\in\GG(\Q_p)$ whose eigenvalues are powers of the
	uniformizer~$\lambda_p=p$ which commutes with~$a$ and for which the eigenvalues
	of~$k_a=a a_{\textrm{nc}}^{-1}$ are of~$p$-adic absolute value one.
	If~$G_{p}<\GG(\Q_p)$ is a subgroup of finite index, and~$a\in G_{p}$ is only assumed to be diagonalizable, 
	we may apply the above to a power of~$a$ and obtain (possibly for a higher power)
	that~$a^\ell=a_{\textrm{nc},\ell}k_{a^\ell}$ as above with~$a_{\textrm{nc},\ell},k_{a^\ell}\in G_{p}$
	for some~$\ell\geq 1$. 

	Combining the archimedean and non-archimedean decompositions we see that there are homomorphisms
	$a_{\textrm{nc}},k : \Z^d \to \GG(\Q_S)$ so that $a_{\textrm{nc}}$ is of class-$\cA'$, the image of $k$ is relatively compact, $a_{\textrm{nc}}(\mathbf n)$ and $k(\mathbf m)$ commute for every $\mathbf n, \mathbf m$, and so that 
	\[
	A_V=\{a_{\textrm{nc}}(\mathbf n)k(\mathbf n): \mathbf n \in \Z^d\}
	\]
	is a finite index subgroup of $A$.
	Let~$A_{\textrm{nc}}=a_{\textrm{nc}}(\Z^d)$, \ $M=\overline{k(\Z^d)} < C_G(A_{\textrm{nc}})$, and~$A'=AM$.

 Let now~$\mu$ be an~$A$-invariant and ergodic joining on~$X$. We define
\[
 \mu'=\int_{M}m.\mu\operatorname{d}\!m_{M}(m),
\]
which is an~$A'$-invariant and ergodic measure. 
As~$A'/A_{\textrm{nc}}$ is a compact and abelian group, it follows
furthermore that a.e.~$A_{\textrm{nc}}$-ergodic components of~$\mu'$
can be obtained from another generic~$A_{\textrm{nc}}$-ergodic component by the action of~$a\in A'$. 
So if we let~$\mu_{\textrm{nc}}$ be a generic~$A_{\textrm{nc}}$-ergodic
component of~$\mu'$, then we have
\begin{equation}\label{muprimeformula}
 \mu'=\int_{A'/A_{\textrm{nc}}}a'.\mu_{\textrm{nc}}\operatorname{d}\!m_{A'/A_{\textrm{nc}}}(a'A_{\textrm{nc}}).
\end{equation}
Theorem~\ref{higher rank} now applies to~$\mu_{\textrm{nc}}$ and~$A_{\textrm{nc}}$. 
It implies that~$\mu_{\textrm{nc}}$
is the~$g^{-1}Hg$-invariant Haar measure~$m_{\Gamma H g}$ on the orbit~$\Gamma H g$ for some finite
index subgroup~$H<\HH(\Q_S)$ of the~$\Q_S$-points of an algebraic group~$\HH<\GG$
defined over~$\Q$. As remarked in the introduction (cf. also Proposition~\ref{joining-groups} below), the joining assumption implies that the group $\HH$ is semisimple.
In particular $A_{\textrm{nc}}<g^{-1}Hg$. 
By Remark~\ref{saturated by unipotent remark}, there is a finite index subgroup $H_1 \leq H$ so that $(\Gamma \cap H_1) \backslash H_1$ is saturated by unipotents. Since $A_{\textrm{nc}}$ was defined only up to a finite index, replacing it with the finite index subgroup $A_{\textrm{nc}}\cap g^{-1}H_1g$ if necessary, we can without loss of generality assume $(\Gamma \cap H) \backslash H$ is saturated by unipotents, and hence by Proposition~\ref{fancy-Mautner}, the action of $A_{\textrm{nc}}$ on $\mu_{\textrm{nc}}=m_{\Gamma H g}$ is weak mixing.

It follows that the action of $\Z^d$ on~$(X,\mu_{\textrm{nc}})\times (M,m_M)$ given by $(x,k) \mapsto (a_{\textrm{nc}}(\mathbf n).x, k(\mathbf n) k)$ for any $\mathbf n \in \Z^d$ is ergodic. The map~$(x,k)\mapsto k.x$ from $X \times M$ to $X$ intertwines the above $\Z^d$ action on $X \times M$ to the $\Z^d$ action corresponding to $A_V$ on $X$, i.e. the $\Z^d$ action given by $x \mapsto a_{\textrm{nc}}(\mathbf n)k(\mathbf n).x$. The image of the product measure $\mu_{\textrm{nc}}\times m_M$ under this map is the measure
\[
 \mu_V=\int_{M}k.\mu_{\textrm{nc}}\operatorname{d}\!m_{M}(k)
\]
and it follows that this measure is ergodic (and invariant) under~$A_V$.

Since~$A'=AM$ and $A\cap M=\emptyset$ we have~$A'/MA_{\textrm{nc}}\cong A/A_V$. Together 
with~\eqref{muprimeformula} we now see that
\[
\mu'=\int_{A'/MA_{\textrm{nc}}}a.\mu_V\operatorname{d}\!m_{A/MA_{\textrm{nc}}}(aMA_{\textrm{nc}})=\int_{A/A_V}a.\mu_V\!\operatorname{d}m_{A/A_V}
\]
is ergodic under the action of~$A$. However, this forces~$\mu'=\mu$ and the corollary follows.
\end{proof}

\subsection{Algebraicity without the eigenvalue assumption}
Corollary~\ref{cor-extension} classifies joinings for actions of many higher rank abelian groups that are not of class-$\cA'$, though the description of the joinings in this case is not as nice as that given in Theorem~\ref{higher rank} for class-$\cA'$ groups. However, there are some cases of groups for which the class-$\cA'$ assumption fails, yet the nice conclusion of Theorem~\ref{higher rank} still hold. 

A typical situation is the following:
Let~$X_1=\Gamma\backslash G$ with $\Gamma=\operatorname{SL}_2(\ZZ[\frac1{pq}])$ diagonally embedded in
$G=\operatorname{SL}_2(\R\times\Q_p\times\Q_q)$
for two primes~$p\neq q$. Let~$a_p\in \operatorname{SL}_2(\ZZ[\frac1{pq}])$ be the diagonal element
with eigenvalues~$p,p^{-1}$ and define~$a_q$ similarly. We let~$A_1=\langle a_p,a_q\rangle$. 
Consider e.g. selfjoinings of the corresponding $\Z^2$-action, i.e.\ we take~$r=2$,~$X=X_1\times X_2$ with~$X_2=X_1$,~$A_2=A_1$, and consider a self-joining~$\mu$
on~$X$. Then once more, Theorem~\ref{higher rank} initially does not apply as the resulting diagonally
embedded subgroup~$A$ is not of class-$\cA'$
(since the $p$-adic component of~$a_q$ belongs to a compact subgroup and is nontrivial). However, 
this example has additional structure that can be exploited: specifically, in this case the ``bad'' $p$-adic component of~$a_q$ 
belongs to the algebraic torus given by the Zariski closure of the ``good'' $p$-adic component of~$a_p$, and similarly for the $q$-adic component of $a_p$ and $a_q$. 

The following is a strengthening of Corollary~\ref{cor-extension},
which applies e.g.\ to the above example.

\begin{corollary}\label{cor-extension-2}
	 Let~$X=X_1\times\cdots\times X_r$ and~$A<A_1\times\cdots\times A_r<G=G_1\times\cdots\times G_r$
	 be as in Corollary~\ref{cor-extension} for some~$r\geq 2$. Suppose in addition that every
	element~$a\in A$ can be written uniquely as~$a=a_{\textrm{nc}}k_a$ where~$a_{\textrm{nc}}$ 
	belongs to a class-$\cA'$ group $A_{\textrm{nc}}<G$
	and~$k_a$ satisfies that for each $\sigma \in S$, the projection of $k_a$ to $\GG(\Q_\sigma)$ belongs to a compact subgroup of the Zariski closure of the projection of $A_{\textrm{nc}}$ to $\GG(\Q_\sigma)$. 
	Then any~$A$-ergodic joining on~$X$ is an algebraic measure defined over~$\Q$.	
\end{corollary}

Given a set $\Omega \subset \GG(\Q_S)$, it will be convenient to refer to the product over all places $\sigma \in S$ of the Zariski closures of the projection of $\Omega$ to $\GG(\Q_\sigma)$ as the Zariski closure of $\Omega$ in $\GG(\Q_S)$.

\begin{proof}
	Let~$\mu$ be an~$A$-ergodic joining. We again apply the proof of Corollary~\ref{cor-extension}. 
	Phrasing our assumptions in the language of this proof
	we have~$A'=A_{\textrm{nc}}M$, where~$M=\overline{\{k_a:a\in A\}}$
	is compact and belongs to the Zariski closure of~$A_{\textrm{nc}}$. Hence we obtain
	from Corollary~\ref{cor-extension}
	an element~$g\in G$, an algebraic group~$\HH<\GG$ defined over~$\Q$, and a finite index subgroup~$H<\HH(\Q_S)$
	containing~$g A_{\textrm{nc}}g^{-1}$,
	which together with~$A'/A$ describe~$\mu$ completely. Clearly~$H$ must contain every
	unipotent subgroup of~$\HH(\Q_S)$. If~$F$ is a non-compact almost direct factor of~$\HH(\Q_S)$
	(defined over~$\Q_\sigma$ for some~$\sigma\in S$), then~$H$ must contain~$F^+$.
	Since~$F/F^+$ is abelian, this implies that~$H$ is normalized by every non-compact factor~$F$
	of~$\HH(\Q_S)$. 

	Since~$A_{\textrm{nc}}<g^{-1}\HH(\Q_S)g$  and~$M$
	belongs to the Zariski closure of~$A_{\textrm{nc}}$ it follows that~$gA'g^{-1}<\HH(\Q_S)$
	belongs to the direct product of the non-compact almost direct factors of~$\HH(\Q_S)$. 
	If necessary we may now replace~$H$ by~$H'=gA'g^{-1}H$, which unlike the general case
	in Corollary~\ref{cor-extension} is now a group. The corollary follows.
\end{proof}

\section{Measure rigidity for unipotent groups}
\label{unipotent measure rigidity}

A basic result in homogeneous dynamics is Ratner's Measure Classification Theorem establishing that any measure on~$X=\Gamma\backslash G$ that is invariant and ergodic \cite{Ratner-Annals} which has been extended to the $S$-arithmetic case by Ratner and by Margulis and Tomanov \cite{Ratner-padic, Margulis-Tomanov}. 

The purpose of this section is to recall this important ingredient in our work in a form adapted to our setting. For instance, in the proof of our main results
we will establish additional invariance of the measure under consideration
by unipotent elements but we will not know in the intermediate steps that the unipotents act ergodically. We will also like to get slightly more precise information on the possible algebraic measures that can appear in the classification that come from the arithmetic structure of the quotient spaces we consider. The following result, obtained by combining the results of~\cite{Margulis-Tomanov-almost-linear} and \cite{Tomanov-orbits} (which are based on the measure classification theorems referred above) will be convenient for our purposes: 

\begin{theorem}\label{MTclassA}
 Let~$X=\Gamma\backslash G$ be an $S$-arithmetic quotient (cf.~\S\ref{intro}) with $G<\GG(\Q_S)$ a finite index subgroup and $\GG$ a connected algebraic group defined over~$\Q$. Let~$A<G$ be a diagonalizable subgroup of class-$\mathcal A'$, and $H_u<G$ 
an~$A$-normalized Zariski-connected
unipotent subgroup generated by Zariski-connected unipotent subgroups. Let~$H=\langle A, H_u\rangle$ be the group generated by~$A$ and~$H_u$. 
Let~$\mu$ be an~$H$-invariant and ergodic probability measure on~$X$. 
Then there exists a connected algebraic~$\Q$-subgroup~$\LL\leq\GG$ and an open finite index 
subgroup~$L\leq\LL(\Q_S)$ such that a.e.~$x\in X$ has a representative~$g\in G$ 
such that the ergodic component~$\mu_x^{\mathcal{E}}$
of~$\mu$ for the point~$x=\Gamma g$ and for the action of~$H_u$ is algebraic and equals the normalized Haar measure~$m_{\Gamma Lg}$
on the closed orbit~$\Gamma L g$. Moreover,~$\mu$ is supported on the closed orbit~$\Gamma N_{G}^1(\LL) g$
for any~$x=\Gamma g\in\supp\mu$. 
\end{theorem}

Here~$N^1_{G}(\LL)$ denotes the group
\[
N^1_{G}(\LL)=\{g\in G: \text{$g$ normalizes~$\LL$ and preserves the Haar measure on~$\LL(\Q_S)$}\}.
\]
Not that in general $N^1_{G}(\LL)$
is \emph{not} the group of~$\Q_S$-points of an algebraic group. In fact, if~$S$ contains more than two places, say~$\infty$ and~$p$, it is possible
for some~$g\in N^1_{G}(\LL)$ to satisfy that each of its components~$g_\sigma$ normalizes~$\LL$ 
or equivalently its Lie algebra~$\mathfrak l$, the real component
has~$\det\left(\operatorname{Ad}_{g_\infty}|_{\mathfrak l}\right)=p>1$ 
and~$\|\det\left(\operatorname{Ad}_{g_p}|_{\mathfrak l}\right)\|_p=1/p<1$
for some prime~$p\in S$. In such cases the algebraic 
group~$N_{\GG}(\LL)$ has a rational character,~$N^1_{G}(\LL)$ contains the~$\Q_S$-points of the kernel
of the character as well as some non-algebraic subgroup of an additional torus subgroup.

Theorem~\ref{MTclassA} differs mildly
from the theorem in~\cite{Margulis-Tomanov-almost-linear} in
the assumptions: We are assuming that the projection of~$A$ to the real component
is contained in the connected component of an~$\R$-diagonalizable subgroup, 
while in~\cite{Margulis-Tomanov-almost-linear} the real component of every element of~$A$ 
must have eigenvalues which are also powers of a single number. This does not affect
the proof of the theorem, see the sketch below. 

The identification of the possible algebraic measures in terms of $\Q$ subgroups $\LL<\GG$ is taken from~\cite{Tomanov-orbits} and allows to simplify the argument of~\cite{Margulis-Tomanov-almost-linear} in the arithmetic case.

\begin{proof}[Sketch of proof of Theorem~\ref{MTclassA}.]
	We recall that the ergodic decomposition
for the action of~$U$ can actually be obtained from the disintegration of~$\mu$
into conditional measures~$\mu=\int\mu_x^{\mathcal{E}}d\mu$,
where~$\mathcal{E}=\{B\in\mathcal{B}_X: \mu(B\triangle u. B)=0$
for all~$u\in H_u\}$ is the~$\sigma$-algebra of~$H_u$-invariant sets.
By Ratner's measure classification the ergodic components~$\mu_x^{\mathcal{E}}$ are 
(wherever they are defined and ergodic) algebraic,
i.e.\ are the Haar measures on closed orbits~$L_x . x$. 

By assumption every~$a\in A$ normalizes~$H_u$ and so~$\mathcal{E}$ is an~$A$-invariant~$\sigma$-algebra.
This in turn implies that~$\mu_{a.x}^{\mathcal{E}}=(a)_*\mu_{x}^{\mathcal{E}}$ a.s.
This shows that~$\operatorname{Ad}_a$ maps the Lie algebra of~$L_x$ onto the Lie algebra of~$L_{a.x}$. The class-$\mathcal A'$
assumption and Poincare recurrence now imply
that a.s.~the Lie algebra of~$L_x$ is normalized by~$a\in A$. By ergodicity under 
the action of~$H=\langle A,H_u\rangle$ it follows
that the Lie algebra of~$L_x$ is a.s.~independent of~$x$.

By~\cite{Tomanov-orbits} there exists for each~$x=\Gamma g$ as above some~$\Q$-group~$\LL_{g}$ such that~$L_x$
is a finite index subgroup of~$g^{-1}\prod_{\sigma\in S}\LL_{g}(\Q_\sigma) g$. 
As there are only countably many algebraic subgroups~$\LL<\GG$ defined over~$\Q$ there exists one such group~$\LL<\GG$
and a set~$B\subset X$ of positive measure such that~$\LL_g=\LL$ for some choice of the representative~$g$
of every element~$x=\Gamma g\in B$. Now choose~$x_0=\Gamma g_0\in B$ such that~$H.x_0$ is dense in~$\supp\mu$ and note 
that the above shows that~$\Gamma N_{G}(\LL) g_0$ is~$H$-invariant and has positive measure.  
Therefore,~$\mu(\Gamma N_{G}(\LL) g_0)=1$. For $\Gamma g \in \Gamma N_{G}(\LL) g_0$ the function sending $\Gamma g$ to the volume of~$\Gamma\LL(\Q_S)g$ (an orbit that we can rewrite as $\Gamma g (g_0^{-1}\LL(\Q_S)g_0)$ if the representative $g$ is chosen out of $N_{G}(\LL) g_0$)
is a finite continuous function. Using Poincare recurrence 
this shows that~$a\in A$ must preserve the Haar measure on~$g_0^{-1}\LL(\Q_S)g_0$
and so~$\mu$ is supported on the orbit~$\Gamma N^1_{G}(\LL) g_0$. 

It remains to show we can find a single finite index subgroup $L <\LL(\Q_S)$ so that  a.s.~the ergodic component~$\mu_x^{\mathcal{E}}$
of~$\mu$ for the point~$x=\Gamma g$ and for the action of~$H_u$ satisfies $\mu_x^{\mathcal{E}}=m_{\Gamma Lg}$. We leave establishing this fact to the reader; the full argument can be found in~\cite{Margulis-Tomanov-almost-linear}. 
\end{proof}

\section{Leaf-wise measures and entropy}\label{sec: conditionalm}

A fairly general construction of leaf-wise measures is presented in \cite[Sect. 3]{Lindenstrauss-03}. A slightly more special construction (sufficient for our purposes) of these measures appear in \cite{Pisa-notes}. Here we  summarize some important properties of these leaf-wise measures.
Mostly we review known facts that can be found e.g.\ in~\cite{Pisa-notes}, though some bits (for instance Proposition~\ref{prop-faithful}) are new.

\subsection{Basic properties}\label{basic properties}

We will be working with
Zariski connected subgroups of $G ^ \Psi$ for $\Psi \subset \Phi $ with $(\Psi + \Psi) \cap \Phi  \subset \Psi$ and $| \alpha(a_0)| < 1$ for some fixed $ a_0 \in A$ and all $\alpha \in \Psi$.
 These are automatically 
Zariski connected unipotent subgroups of~$G$.
If a Zariski connected unipotent subgroup $U<G$ is normalized by $A$ then for every $x  \in X$ and $a  \in
A$, $a.(U. x) =x U a^{-1} =x a^{-1} U = U.(a.x)$, so that the {\em foliation} of $X$ into $U$-orbits is invariant under the action of $A$. We will say that $a
\in A$ {\em expands} ({\em contracts}) the $U$-leaves, or simply $U$, if the absolute values of all eigenvalues (using of course~$|\cdot|_\sigma$ for the eigenvalues in~$\Q_\sigma$) of $ \operatorname{Ad} _ a$ restricted to the Lie algebra of $U$ are greater (smaller) than one.

For any locally compact metric space $Y$ let $\mathcal{M} _ \infty (Y)$ denote
the space of Radon measures on $Y$ equipped with the weak$^*$ topology, i.e.\
all locally finite Borel measures on $Y$ with the coarsest topology for which
$\rho \mapsto \int _ Y f (y) \operatorname{d}\! \rho (y)$ is continuous for every compactly
supported continuous $f$. For two Radon measures $ \nu_1$ and $ \nu_2$ on $Y$ we
write
\begin{equation*}
\nu_1\propto \nu_2\mbox{ if } \nu_1=C \nu_2\mbox{ for some }C>0
\end{equation*}
and say that $\nu_1$ and $\nu_2$ are proportional.

We let $B^Y_ \epsilon (y)$ (or $B _  \epsilon (y)$ if $Y$ is understood) denote
the ball of radius $ \epsilon$ around $y  \in Y$; if $H$ is a group we set $B ^
H _ \epsilon = B ^ H _ \epsilon (e)$; and if $H$ acts on $X$ and $x  \in X$ we let $B ^ H _ \epsilon (x) = B ^ H _ \epsilon .x$.

Let $ \mu$ be an $A$-invariant probability measure
on $X$. For any Zariski connected unipotent subgroup $U < G^ \Psi < G$ normalized by $A$, one
has a system  $ \left\{ \mu_{x,U} \right\} _ {x  \in X}$ of Radon
measures on $U$ and a co-null set $X' \subset X$ with the
following properties:
\begin{enumerate}[ref=(\arabic*)]
\item The map $x \mapsto  \mu _ {x, U}$ is measurable.
\item For every $\epsilon > 0$ and $x  \in X '$ it holds that $ \mu _ {x, U}
(B ^ U _  \epsilon)>0$.
\item\label{5.1.3} For every $x  \in X'$ and $u  \in U$
with $u.x  \in X '$, we have that $ \mu _ {x,U} \propto ( \mu _ {u.x,
U}) u$, where  $( \mu _ {u.x, U}) u$ denotes the push forward of the
measure $ \mu _ {u.x, U}$ under the map $v \mapsto vu$.
\item\label{5.1.4}
For every $a  \in A$, and $x, a .x \in X'$, $ \mu _
{ a .x, U} \propto  (\theta_a)_* ( \mu _{x,U})$ where~$\theta_a:U\to U$ is the conjugation
defined by~$\theta_a(u)=aua^{-1}$.
\end{enumerate}
In general, there is no canonical way to normalize the measures $ \mu _{x,U}$;
we fix a specific normalization by requiring that $ \mu _{x,U} (B ^ U _ 1) = 1$
for every $x  \in X'$. This implies the next crucial property.
\begin{enumerate}[resume,ref=(\arabic*)]
\item \label{5.1.5} If $U \subset C( a )=\{g \in G _ S: g a= a g\}$
is centralized by $ a \in A$, then
$ \mu_ { a.x, U} =  \mu _{x,U}$ whenever $x, a .x \in X'$.
\item $ \mu$ is $U$-invariant if, and only if, $ \mu _ {x, U}$ is
a Haar measure on $U$ a.e.\ (see \cite[Prop.~4.3]{Lindenstrauss-03}).
\end{enumerate}

The other extreme to invariance as above is where $ \mu _ {x, U}$ is atomic. If
$  \mu$ is $A$-invariant then outside some set of measure zero if $ \mu _ {x, U}$
is atomic then it is supported on the single point $ e  \in U$, in which case
we say that {\em $ \mu _ {x, U}$ is trivial}.

The leaf-wise measures for the Zariski connected unipotent subgroup $ G ^ {[\alpha]} $ associated to a Lyapunov weight
 $ \alpha \in \Phi  $ we denote by $ \mu _ x ^{ [\alpha]}$, and more generally we write $\mu _ x ^ \Psi$ for the leaf-wise measures on $G ^ \Psi$ when $\Psi \subset \Phi $ is a set of Lyapunov weights such that $G ^ \Psi$ is a Zariski connected unipotent subgroup.
Note that if~$d\geq 2$ and~$\alpha$ is a nonzero Lyapunov weight and
 $a  \in A$ is such that $|\alpha (a) |=1 $, then we have for
 the leaf-wise measures $ \mu_x^{[\alpha]}$ that property~\ref{5.1.5} above holds.
 Of particular importance to us is the following characterization of
positive entropy.
\begin{enumerate}[resume,ref=(\arabic*)]
\item\label{5.1.7} Let $\Psi = \{ \alpha \in \Phi _ S: \alpha (a)<0\}$ so that $ G ^ \Psi=G_a^-$ is the {\em horospherical stable subgroup} defined by $a$. Then the measure theoretic entropy
$ \operatorname{h}_ \mu( a)$ is positive if and only if the leaf-wise measures
$ \mu_{x,G_a^-}$ are nonatomic a.e.
\end{enumerate}

\subsection {Entropy contribution } \label{Entropy contribution}
In this section we refine property~\ref{5.1.7} above to a more quantitative statement. For more details see \cite[Sect.~9]{Margulis-Tomanov} and \cite[Sect.~9]{EinsiedlerKatokNonsplit}.

Let $U$ be a Zariski connected unipotent subgroup normalized by $A$ such that $a \in A$ contracts  $U$. Then for any
$a$-invariant probability measure $\mu $ on $X $ the limit
\begin{equation}\label{volume decay}
\vol_\mu(a, U,x)=-\lim_{n\rightarrow\infty}
\frac{\log\mu_{x,U} \bigl(\theta_a^n(B_1^ U)\bigr)}{n}
\end{equation}
exists for a.e.\ $x\in X $ by \cite[Lemma 9.1]{EinsiedlerKatokNonsplit}.
If furthermore, $\mu_{x, U }$ is supported
by a Zariski connected unipotent subgroup $P\subseteq U $ that is  normalized by $A$, then
\begin{equation}\label{module bound}
\vol_\mu(a,U,x)\leq
\md(a,P)=\sum_{\alpha\in\Psi}
\alpha(a)^-\dim\bigl(\fp\cap\fg^\alpha\bigr)
\end{equation}
for a.e.\ $x\in X$. Here we write $r^-=\max(0,-r)$ for the negative part of $r\in\R$, $\fp$ for the Lie algebra of $P$ and~$\dim\bigl(\fp\cap\fg^\alpha\bigr)$ denotes the sum of the dimensions~$\dim_{\Q_\sigma}\bigl(\fp\cap\fg^\alpha\cap\mathfrak g_\sigma\bigr)$ for all~$\sigma\in S$, where~$\mathfrak g_\sigma$ denotes the Lie algebra of~$\GG(\Q_\sigma)$. 
 Note that $\mathfrak p$ is a direct sum
of its subspaces $\mathfrak p \cap \mathfrak g ^ \alpha$,
since $P$ is normalized by $A$.  In fact, $\md(a,P)$ is the negative logarithm of the module of the restriction of $\theta_a$ to $P$.
It is easy to check that $\vol_\mu(a,U,\cdot)$ is $A$-invariant,
and so constant for an $A$-ergodic measure. We write $\operatorname{h} _ \mu (a, U) $
for the integral of $\vol_\mu(a,U,\cdot)$, and will refer to it as the {\em entropy contribution of $U$}.

A $\sigma$-algebra $\cA$ of Borel subsets of
$ X$ is {\em subordinate to $U$} if $\cA$ is countably generated,
for every $ x \in  X$ the atom $[ x]_\cA$ of $ x$ with respect to
$\cA$ is contained in the leaf $ U.x$, and for a.e.\ $ x$
\[
B_\epsilon^U .x\subseteq [ x]_\cA\subseteq
B_\rho^U.x\mbox{ for some } \epsilon>0\mbox{ and
}\rho>0.
\]
The conditional measures~$\mu_x^{\cA}$ for a ~$\sigma$-algebra~$\cA$ which is subordinate
to~$U$ are strongly related to the leaf-wise measure~$\mu_{x,U}$. Indeed, if~$V_x\subset U$
describes the atom in the sense that~$V_x.x=[x]_\cA$ then the conditional measure
can be obtained by the push-forward 
\begin{equation*}
  \mu_x^\cA\propto  \left(\mu_{x,U}|_{V_x}\right).x
\end{equation*}
of the leaf-wise measure~$\mu_{x,U}$ restricted to~$V_x$.

A $\sigma$-algebra $\cA$ is {\em $a$-decreasing} if
$a^{-1}\cA \subseteq\cA$.
By \cite[Lemma 9.3]{EinsiedlerKatokNonsplit} we have
\begin{equation}\label{eq:entropy and algebras}
\Hh_\mu(\cA|a^{-1}\cA)=\int\vol_\mu(a,U,x)\operatorname{d}\!\mu.
\end{equation}
whenever $\cA$ is an $a$-decreasing $\sigma$-algebra that is subordinate to $U$.

A key technical fact which allows us to show the entropy contribution for $G ^ - _ a$ of $\mu$ coincides with the entropy is the following result, which has a long history in smooth dynamics and in our context can be found in \cite{Margulis-Tomanov}:

\begin{proposition} [{\cite[Prop. 9.2]{Margulis-Tomanov}}]\label{generator proposition}
Let $\mu$ be an $a$-invariant and ergodic measure on $X = \Gamma \backslash G$. Then there exist a countable partition $\mathcal{P}$ with finite entropy which is a generator for $a$ modulo $\mu$. Moreover, the $\sigma$-algebra $\cA = \bigvee_ {n \geq 0} a ^ {- n} \mathcal{P}$ is $a$-decreasing and subordinate to $G ^ - _ a$.
\end{proposition}

For proof, see \cite{Margulis-Tomanov} or \cite[Prop. 7.43]{Pisa-notes}. 

If $\cP$ is generating for $a$ mod $\mu$, and $\cA = \bigvee_ {n \geq 0} a ^ {- n} \mathcal{P}$ then clearly  $\operatorname{h} _ \mu (a) = \operatorname{H} _ \mu (\cA | a^{-1}\cA)$. This implies by \eqref{eq:entropy and algebras}, at least for $a$-ergodic ergodic $\mu$ that 
\begin{equation}\label{entropy contribution of horospherical}
\operatorname{h} _ \mu (a) =h_\mu(a,G^-_a)= \int\vol_\mu(a,G_a^-,x)\operatorname{d}\!\mu;
\end{equation}
and can be derived using the ergodic decomposition from the ergodic case also for $\mu$ not $a$-ergodic, see \cite[Prop.~9.4]{EinsiedlerKatokNonsplit} or \cite[Sect. 7.24]{Pisa-notes} for details.

Using the partition $\cP$ constructed in Proposition~\ref{generator proposition} for $\mu$ ergodic, given  a subgroup $U < G ^ - _ a$ one can construct a countably generated $\sigma$-algebra $\mathcal{A} _ U$ subordinate to $U$ from the $G ^ - _ a$-subordinate decreasing $\sigma$-algebra $\mathcal{A} = \bigvee_ {n \geq 0} a ^ {- n}  \mathcal{P}$ in such a way that for $\mu$-a.e.~$x \in X$ the $\mathcal{A} _ U$ atom of $x$ is related to that of $\mathcal{A}$ by
\begin{equation*}
[x]_{\mathcal{A} _ U} = [x]_{\mathcal{A}} \cap U.x\ 
\end{equation*}
(this condition determines $\mathcal{A} _ U$ up to $\mu$-null sets, but something needs to be said to construct such a countably generated $\sigma$-algebra $\mathcal{A} _ U$.) The important properties of $\mathcal{A} _ U$ for our purposes are the following:
\begin{enumerate}[label=\emph{(\roman*)},align=left,labelindent=\parindent,leftmargin=*,ref=(\roman*)]
\item \label{A-U 1}
$\mathcal{A} _ U$ is subordinate to $U$
\item \label{A-U 2} $\mathcal{A} _ U$ is $a$-decreasing
\item \label{A-U 3}$\mathcal{A} _ U = \mathcal{P} \vee a ^{-1} \mathcal{A}_U$
.\end{enumerate}
For details and proofs of the above claims, see \cite[Sect.~7]{Pisa-notes}, especially \cite[Prop.~7.37]{Pisa-notes}.

\begin{lemma}\label{lemma-1-faithful}
	 Let~$X=\Gamma\backslash G$ and~$a\in G$ be as above. 
	Let~$V<U\leq G_a^-$ be two~$a$-normalized Zariski connected unipotent subgroups. Let~$\mu$
 be an~$a$-invariant probability measure on~$X$. Then~$h_\mu(a,V)\leq h_\mu(a,U)$.
Moreover, the following are equivalent;
\begin{enumerate}[align=left,labelindent=\parindent,leftmargin=*,ref=(\arabic*)]
	\item\label{5.2.1} $h_\mu(a,V)= h_\mu(a,U)$,
	\item\label{5.2.2} $\mu_{x,U} (V)>0$ a.s., and
	\item\label{5.2.3} $\mu_{x,U}=\mu_{x,V}$ a.s.
\end{enumerate}
\end{lemma}

\begin{proof}
	Assume first~$\mu$ is~$a$-ergodic. 
	Let~$\mathcal{P}$ be a countable generator for~$a$ with 
	respect to~$\mu$ as in Proposition~\ref{generator proposition} and $\cA_U, \cA_V$ $a$-decreasing $\sigma$-algebras subordinate to $U$ and $V$ respectively satisfying \ref{A-U 1}--\ref{A-U 3}; it follows that $\cA_U \subset \cA_V$.
Therefore,
\[
h_\mu(a,V)=H_\mu(\mathcal{P}|a^{-1}\mathcal{A}_V)
\leq H_\mu(\mathcal{P}|a^{-1}\mathcal{A}_U)=h_\mu(a,U)
\]
 by the monotonicity
of the entropy function.

It is clear that \ref{5.2.3} implies both~\ref{5.2.1} and~\ref{5.2.2}.
So suppose now equality of the entropy contributions holds as in~\ref{5.2.1}. Then 
\begin{align*}
 (n+1)h_\mu(a,U)&=(n+1)h_\mu(a,V)=H_\mu(\mathcal{P}_{-n}^0|\mathcal{A}_V)=
                 H_\mu(\mathcal{P}_{-n}^n|\mathcal{A}_V)\\
&=H_\mu(\mathcal{P}_{-n}^0|\mathcal{A}_U)=
								                 H_\mu(\mathcal{P}_{-n}^n|\mathcal{A}_U),
\end{align*}
where we write~$\mathcal{Q}_k^\ell=\bigvee_{n=k}^{\ell}a^{-n}.\mathcal{Q}$
for the dynamical refinement of a~$\sigma$-algebra~$\mathcal{Q}$ whenever~$k\leq\ell$.
Now fix some finite partition~$\mathcal{Q}\subset\mathcal{A}_V$ and notice that
\[
 H_\mu(\mathcal{P}_{-n}^n|\mathcal{A}_V)=
                 H_\mu(\mathcal{P}_{-n}^n|\mathcal{Q}\vee\mathcal{A}_U)=
								 H_\mu(\mathcal{P}_{-n}^n|\mathcal{A}_U)
\]
(where we first have inequalities which in light of the above must be equalities).
Therefore,
\begin{align*}
 H_\mu(\mathcal{P}_{-n}^n\vee\mathcal{Q}|\mathcal{A}_U)&=
 H_\mu(\mathcal{Q}|\mathcal{A}_U)+
 H_\mu(\mathcal{P}_{-n}^n|\mathcal{Q}\vee\mathcal{A}_U)\\
 &=H_\mu(\mathcal{P}_{-n}^n|\mathcal{A}_U)+
 H_\mu(\mathcal{Q}|\mathcal{P}_{-n}^n\vee\mathcal{A}_U),	
\end{align*}
which gives~$H_\mu(\mathcal{Q}|\mathcal{A}_U)=
H_\mu(\mathcal{Q}|\mathcal{P}_{-n}^n\vee\mathcal{A}_U)$.
As~$\mathcal{P}$ is a generator, we have~$\mathcal{P}_{-n}^n\nearrow\mathcal{B}_X$ (the Borel $\sigma$-algebra on $X$) as~$n\to\infty$. Therefore,~$\mathcal{Q}\subset\mathcal{A}_U$ modulo~$\mu$.
As~$\mathcal{Q}\subset\mathcal{A}_V$ was an arbitrary finite partition
we see that~$\mathcal{A}_V=\mathcal{A}_U$ modulo~$\mu$ and so~$\mu_x^{\mathcal{A}_U}=\mu_x^{\mathcal{A}_V}$ a.s.
Since conditional measures characterize leafwise measures a.s.~we get~$\mu_{x,V}=\mu_{x,U}$
a.e.\ as in \ref{5.2.3} (first when restricted to a neighborhood of the identity 
and then by using Property~\ref{5.1.4}
from \S\ref{basic properties} for~$a^{-1}$ without restrictions). 

Assume now that~\ref{5.2.2} holds and let~$\mathcal{A}_U$ and~$\mathcal{A}_V$ be as above.
As~$\mu_{x,U}$ is locally finite and we assume that~$\mu_{x,U}(V)>0$ for a.e.~$x$,
it follows that~$\mu_{x,U}$ is supported on countably many 
subsets~$Vu_1,Vu_2,\ldots$ for~$\{u_1,u_2\ldots\}\subset U$. This shows that a.s.
every atom of the~$\sigma$-algebra~$\mathcal{A}_U$ splits, up to a nullset,
into at most countably many
atoms for the~$\sigma$-algebra~$\mathcal{A}_V$. However, in this case
the conditional measure for~$\mathcal{A}_V$ can easily be obtained from the conditional 
measure for~$\mathcal{A}_U$; simply by restricting to the correct atom and renormalizing
the resulting measure to be a probability measure. Translating this statement to the leafwise
measure we see that~$\mu_{x,V}$ is obtained by restricting~$\mu_{x,U}$ to~$V$ 
and renormalizing it to satisfy the property~$\mu_{x,V}(B_1^V)=1$. Therefore, there exists a constant~$c_x$ such that
\[
 \mu_{x,V}(a^n B_1^V a^{-n})\leq c_x \mu_{x,U}(a^n B_1^U a^{-n}).
\]
Now take the logarithm, divide by~$-n$, and take the limit to obtain~$\vol_\mu(a,V,x)\geq\vol_\mu(a,U,x)$ a.s. Taking the integral we obtain~$h_\mu(a,V)=h_\mu(a,U)$, i.e.~\ref{5.2.1} which implies also~\ref{5.2.3}.

If~$\mu$ is not ergodic with respect to~$a$ we take its ergodic decomposition, which can be obtained by using the conditional measures~$\mu_x^{\mathcal{E}}$ for the~$\sigma$-algebra
$\mathcal{E}=\{B\subset X: B=a.B\}$. For a.e.~$x$ the leafwise measure of the ergodic component~$\mu_x^{\mathcal{E}}$ satisfies~$(\mu_x^{\mathcal{E}})_{y,U}=\mu_{y,U}$ for~$\mu_x^{\mathcal{E}}$-a.e.~$y$ (see e.g.~\cite[Prop.~7.22]{Pisa-notes}), and hence the ergodic case of the lemma implies the
general one.
\end{proof}

\subsection{A corollary of the eigenvalue assumption regarding the support of leafwise measures}\label{cor-section}

\begin{lemma}\label{cor-eigenvalue}
	Let~$X=\Gamma\backslash G$ and~$A<G$ be as in Theorem~\ref{higher rank}. Let~$\mu$
	be an~$A$-invariant measure, let~$a_0\in A$, and let~$U\leq G_{a_0}^-$ be an~$A$-normalized Zariski
	connected unipotent subgroup. 
	For~$x\in X'$ define~$P_x^U\leq U$ to be the smallest Zariski connected closed subgroup that contains~$\supp\mu_{x,U}$. 
	Then for a.e.~$x\in X'$ the subgroup~$P_x^U$ is~$A$-normalized and~$P_{a.x}^U=P_x^U$ for all~$a\in A$, hence for $\mu$ $A$-ergodic is equal a.e.~to a single $A$ normalized group $P^U \leq U$.
\end{lemma}

\begin{proof}
  We outline the argument, for more details see e.g.~\cite[Sect.~6]{EinsiedlerKatokNonsplit} (where a similar
statement is proven under weaker assumptions for the maximal subgroup of~$U$ that leaves~$\mu_{x,U}$ invariant).
 Since~$(\theta_{a})_*\mu_{x,U}=\mu_{a.x,U}$ for a.e.~$x$ by~(4) in \S\ref{basic properties} 
we also have~$P_{a.x}^U=aP_x^U a^{-1}$. Now let~$K\subset X'$ be a Luzin set of almost full measure on
which~$x\in K\mapsto\mu_{x,U}$ is continuous. Fix some~$a\in A$ that contracts~$U$. It follows
that for a.e.~$x\in K$ there exists~$n_k\to\infty$ with~$a^{n_k}.x\to x$ as~$k\to\infty$. 
By continuity we see that~$\mu_{x,U}$ is also supported on any limit point of~$a^{n_k}P_x^Ua^{-n_k}$ (where we
take the limit within the product of the appropriate Grassmannian varieties) -- 
by equality of the dimensions we see that~$a^{n_k}P_x^U a^{-n_k}$
converges to~$P_x^U$. However, as $a \in A$ and $A$ is of class-$\mathcal A'$ any limit point of~$a^{n_k}P_x^Ua^{-n_k}$ is automatically normalized by~$a$. 
As~$a\in A$ was an arbitrary element of an open cone in $A$,
we see that~$P_x^U$ is~$A$-normalized. Increasing the Luzin set~$K$ the statement
follows a.e.
\end{proof}


\subsection{Faithful recurrence}
Let $X=\Gamma \backslash G$ be an $S$-arithmetic quotient, $\mu$ a probability measure on $X$, and $H < G$ an algebraic subgroup (i.e.~a product of algebraic subgroups over $\Q_s$ for every $s \in S$).
Recall that recurrence of $H$
with respect to $\mu$ is equivalent to $ \mu _ {x,H }(H) = \infty$ by \cite [Prop.~4.1]{Lindenstrauss-03}. 

\begin{definition}
We say $\mu$ is {\em faithfully $H$-recurrent} if it is $H$-recurrent and for almost every $x$ there does not exist a proper Zariski closed subgroup $H' < H$ such that $\mu _ {x, H}$ is supported on~$H '$. 
\end{definition}

\begin{remark}\label{faithful remark}
 For $U$ and $\mu$ as in \S\ref{cor-section}, faithful $U$-recurrence is equivalent to $P^U=U$.
 \end{remark}
 
 The only non-trivial part is in this statement is to show that $\mu$ is $U$-recurrent, which in view of \cite [Prop.~4.1]{Lindenstrauss-03} follows from the fact that for $a$-invariant $\mu$ for a.e.~$x$ where $\mu_{x,U}$ is finite in fact $\mu_{x,U}$ is the trivial measure (see \cite[Thm.~7.6.(iii)]{Pisa-notes}).

\begin{proposition}\label{prop-faithful}
	 Let~$X=\Gamma\backslash G$ be an~$S$-arithmetic quotient. Let~$a\in G$ be of class-$\cA'$ and let $U<G_a^-$ be an~$a$-normalized
	Zariski connected unipotent subgroup. Suppose~$\mu$
	is an~$a$-invariant probability measure on~$X$. Then $\mu$ is faithfully~$U$-recurrent iff
	for every measurable~$B\subset X$, every proper~$a$-normalized subgroup~$V<U$,
	and a.e.~$x\in B$ there exists a sequence~$u_k\in U$ with~$u_k.x\in B$ for all~$k$,
	and~$u_kV$ going to infinity in~$U/V$ as~$k\to\infty$.
\end{proposition}

\begin{proof}
  The ``if'' direction follows from property \ref{5.1.4} of \S\ref{basic properties} by chosing $B$ to be the co-null set $X'$ of that section.
  
  For the other direction, we note first that using the ergodic decomposition it suffices to consider the case of $\mu$ ergodic under $a$. Let~$\mathcal{A}_U$ be an~$a$-decreasing~$U$-subordinate~$\sigma$-algebra on~$X$ as in \S\ref{Entropy contribution}. By assumption~$a^{-n}.\mathcal{A}_U\searrow\mathcal{A}_\infty$
for a~$\sigma$-algebra~$\mathcal{A}_\infty$. Therefore the decreasing Martingale convergence theorem
gives a.s.
\begin{equation}\label{martingale equation}
 \mu_x^{a^{-n}.\mathcal{A}_U}(B)(x)=E(1_B|a^{-n}.\mathcal{A}_U)\to
 E(1_B|\mathcal{A}_\infty)(x)=\mu_x^{\mathcal{A}_\infty}(B)
\end{equation}
as~$n\to\infty$. Recall also that for a.e.~$x\in B$ we have~$\mu_x^{\mathcal{A}_\infty}(B)>0$.

By the assumption of faithful recurrence and Lemma~\ref{lemma-1-faithful} 
we have~$\mu_{x,U}(V)=0$ and so for every~$\varepsilon>0$ there exists some~$\rho>0$ such that
\[
 Y_\varepsilon=\{x\in X:  \mu_x^{\mathcal{A}_U}(B_\rho^U V . x)<\varepsilon \}
\]
has measure~$\mu(Y_\varepsilon)>1/2$. Applying the ergodic theorem we get a set~$X'$ of full measure
such that for all~$x\in X'$ the above pointwise Martingale convergence holds
and there exists infinitely many~$n\geq 0$ with~$a^n.x\in Y_\varepsilon$. 

Suppose now~$x\in X'$ and~$\mu_x^{\mathcal{A}_\infty}(B)>\varepsilon$. Then by \eqref{martingale equation} and the definition of~$X'$ we can find infinitely many~$n\geq 0$
with~$\mu_x^{a^{-n}.\mathcal{A}_U}(B)>\varepsilon$ and~$a^n.x\in Y_\varepsilon$. The latter gives
\[
 \mu_x^{a^{-n}.\mathcal{A}_U}((a^{-n}B_\rho^U V a^n ).x)=\mu_{a^n.x}^{\mathcal{A}_U}(B_\rho^UV.(a^n.x))<\varepsilon.
\]
It follows that there exists a~$u\in U\setminus(a^{-n}B_\rho^UVa^n)$ with~$u.x\in B$.
As this holds for infinitely many~$n$ the proposition follows for~$x$. Letting~$\varepsilon$ go to zero,
we obtain the result as required for a.e.~$x\in B$.
\end{proof}

\section{Finer structure of the leaf-wise measures}\label{fine structure}

\subsection{The structure of leaf-wise measures for the horospherical subgroup and the high entropy method} \label{product structure}
$\nope$ 

Fix some $a\in A$, and let $\Psi =\{ \alpha \in \Phi : G ^ {[\alpha]} \subset G_a^-\}$ (recall that $G ^{ [\alpha]}$ is the coarse Lyapunov subgroups corresponding to~$[\alpha]$).  
The leafwise measure $\mu_{x,G_a^-}$ are determined by the leafwise measures on the leaves of the coarse Lyapunov folliations. We recall here the structure theorem describing this relation \cite[Thm.~8.4--8.5]{EinsiedlerKatokNonsplit} in a form that will be convenient in the case of an acting group of class-$\mathcal A'$.

\begin{theorem}[{\cite[Thm.~8.4]{EinsiedlerKatokNonsplit}}]\label{EKtheorem}
Fix an arbitrary order $\{ [\alpha _ 1],\ldots, [\alpha _ \ell]\}$ on the coarse Lyapunov weights $\{[\alpha]: \alpha\in\Psi\}$
and define $\phi: \prod_ {i=1}^\ell G ^ {[\alpha_i]} \rightarrow G_a^-$ by $\phi (u_1,\ldots,u_\ell)=u_1\cdots u_\ell$ for any $(u_1,\ldots,u_\ell) \in \prod_ {i=1}^\ell G ^ {[\alpha_i]}$.  Then
\begin{equation}\label{mu product}
\mu_{x,G_a^-}\propto
 \phi_*\bigl(\mu_x^{[\alpha_1]}\times\cdots\times\mu_x^{[\alpha_\ell]}\bigr)
\mbox{ a.e.}
\end{equation}
More generally for any Zariski connected~$A$-normalized subgroup~$U<G_a^-$ one has
\begin{equation}\label{mu product generalized}
\mu_{x,U}\propto
 \phi_*\bigl(\mu_{x, U\cap G^{[\alpha_1]}}\times\cdots\times\mu_{x, U\cap G^{[\alpha_\ell]}}\bigr)
\mbox{ a.e.}
\end{equation}
\end{theorem}

Equations~\eqref{mu product} and \eqref{volume decay} implies that the following relation between the entropy and the contributions of each of the coarse Lyapunov weights:
\begin{equation} \label{Sum formula of entropy}
\operatorname{h} _ \mu (a) = \operatorname{h} _ \mu (a, G_a^-) = \sum_ {i=1}^\ell \operatorname{h} _ \mu (a,G ^{[\alpha _ i]})
\end{equation}
(cf.~also \cite{Hu-commuting-diffeomorphisms} for a closely related result about the entropy of commuting diffeomorphisms). Similarly, we have for any $U<G_a^-$ as in Theorem~\ref{EKtheorem} that
\begin{equation} \label{Sum formula of U entropy}
\operatorname{h} _ \mu (a, U)  = \sum_ {i=1}^\ell \operatorname{h} _ \mu (a,U \cap G ^{[\alpha _ i]})
.\end{equation}
In particular, as observed already in \cite{KatokSpatzier96}, positive entropy implies that there exists a coarse Lyapunov weight $[\alpha]$ such that the leaf-wise measure $\mu _ x ^{ [\alpha]}$ is nontrivial a.e. 

The argument establishing Theorem~\ref{EKtheorem} 
also establishes another important feature of the leafwise measure
(cf.~\cite[Cor.~6.5]{Lindenstrauss-03} and~\cite[Cor.~8.6]{Pisa-notes}). For this suppose
that~$H<G$ commutes with~$a$ and with~$G_a^-$.
Then there exists a set of full measure~$X'$ such that
\begin{equation}\label{mu product corollary}
 \mu_{x,G_a^-}=\mu_{h.x,G_a^-}\mbox{ whenever }h\in H\mbox{ and } x,h.x\in X'. 
\end{equation}

When the $G^{[\alpha_i]}$ commute which each other the ordering chosen for the coarse Lyapunov weights is not important. But when there are at least two noncommuting coarse Lyapunov groups contained in a single horospheric group, with the corresponding leafwise measures both non-trivial, for \eqref{mu product} to hold the leafwise measures $\mu _ {x,G_a^-}$ and hence $\mu$  must be invariant under a nontrivial unipotent group. This key observation is at the heart of what is known as the \textbf{high entropy method} of producing unipotent invariance:

\begin{theorem}[{\cite[Thm.~8.5]{EinsiedlerKatokNonsplit}}]\label{high entropy}
Let $\mu$ be an $A$-invariant and ergodic probability measure
on $X=\Gamma\backslash G$. Let $a\in A$.
Then there exist two connected
$A$-normalized subgroups $H\leq P\leq G_a^-$ such that
\begin{enumerate}
\item $\mu_{x,G_a^-}$ is supported by $P$ a.e.
\item $\mu_{x,G_a^-}$ is left- and right-invariant
under multiplication with elements of $H$.
\item $H$ is a normal subgroup of $P$ and any elements
$g\in P\cap G^{[\alpha]}$ and $h\in P\cap G^{[\beta]}$
of different coarse Lyapunov subgroups ($[\alpha]\neq[\beta]$)
satisfy that $g H$ and $h H$ commute with each other in
$P/H$.
\item $\mu_x^{[\alpha]}$ is left- and right-invariant
under multiplication with elements of $H\cap U^{[\alpha]} $ for all $\alpha$.
\end{enumerate}
\end{theorem}

\subsection{Alignment and the low entropy theorem}\label{low-entropy}

We now recall from \cite{Low-entropy} our second
tool to show additional unipotent invariance, which is known as the low entropy method.

Let us start with the necessary terminology.
Let $X$ be a locally compact metric space with a Borel probability measure $\mu$,
let $a: X \rightarrow X$ be a measure preserving transformation of $(X, \mu)$,
let $H$ be a locally compact metric group acting continuously and locally free on $X$. We denote the action by $h.x$ for $h\in H$ and $x\in X$.
Let $F:X \rightarrow Y$ be a measurable map to a Borel space $Y$.
We say $\mu$ is $(F, H)$-{\em transient} if there exists a set $X ' \subset X$ of full measure such that there are no two different $x, y \in X '$ on the same $H$-orbit $H.x=H.y$ with $F (x) = F (y)$.
We say $\mu$ is {\em locally $(F, H)$-aligned} if for every $\epsilon > 0$ and compact neighborhood $O$ of the identity in $H$ there exists $B \subset X$ with $\mu (B) > 1 - \epsilon$ and some $\delta > 0$ such that $x, y \in B$
with distance $d (x, y) < \delta$ and $F (x) = F (y)$ implies that $y = h.x$
for some $h \in O$.

In the setting of Theorems~\ref{higher rank}--\ref{thm:perfect} we fix some $a\in A$ (which
we will specify later).    We assume that $a$ uniformly contracts the leaves of the foliation defined by the $U$-orbits of some nontrivial closed unipotent subgroup $U < G$. In other words that $ a$ normalizes $U$ and $U<G_a^-$.

In \cite[Thm.~1.4]{Low-entropy} we have shown the following dichotomy.

\begin{theorem}  \label{low entropy theorem}
Suppose~$G<\GG(\Q_S)$ is a finite index subgroup of an algebraic group~$\GG$ defined over~$\Q$,~$\Gamma<G$
a discrete subgroup, $X = \Gamma \backslash G$ and that~$a\in G$ is of class-$\mathcal A'$.
Let~$U<G_a^-$ be an~$a$-normalized Zariski connected unipotent subgroup.
Assume in addition that for all nontrivial $g \in G ^ +_a$ there exists some $ u \in U$ with $u g u ^{- 1} \not \in G ^ + _aG ^ 0_a$.
Then any $a$-invariant and faithfully $U$-recurrent probability measure $\mu$ on $X$ satisfies at least one of the following conditions:
\begin{enumerate}[ref=(\arabic*)]
\item\label{not transient item} $\mu$ is not $(\mu _ {x, U}, C_G(U) \cap G ^ -_a)$-transient.
\item\label{aligned item} $\mu$ is locally $(\mu _ {x , U},C_G(U)\cap C_G(a))$-aligned.
\end{enumerate}
\end{theorem}

\subsection{On the support of the leafwise measures for joinings}\label{support of joinings}
The results described in \S\ref{product structure} and \S\ref{low-entropy} highlight the central role of the Zariski closure of the support of leafwise measures. For the special case of joinings, the assumption that the measure projects to the uniform measure on each factor gives rise to significant restrictions on the support, as was observed and exploited in a previous paper \cite{Einsiedler-Lindenstrauss-joining}, specifically in \cite[\S3]{Einsiedler-Lindenstrauss-joining}. In that paper, this information was only used in conjunction with the high entropy techniques presented in \S\ref{product structure}; being able to use them with powerful toolbox for analyzing measures invariant under diagonalizable actions, and in particular \S\ref{low entropy theorem} is key to why the results of this paper are less restrictive with regards to the choice of the acting group than that of e.g. \cite{full-torus-paper}.

It would be useful for us to obtain slightly more general formulation of these facts that is applicable e.g. in the context of perfect groups as considered in Theorem~\ref{thm:perfect}, and for this reason (as well as for the convenience of the reader) we reproduce the argument here.

\begin{proposition}[cf.~{\cite[Prop. 3.1]{Einsiedler-Lindenstrauss-joining}}] \label{entropy inequality proposition}
Let $G$ and $\Gamma$ be as above. Let $a \in G$ and $\mu$ be an $a$-invariant measure on $X = \Gamma \backslash G$. Let $N \lhd G$ with $N \cap \Gamma$ a lattice in $N$, $G_1 = G/N$, $\pi:G \to G _ 1$ the natural projection, $a _ 1 = \pi (a)$. Let $\Gamma _ 1 = \pi (\Gamma)$ and set $X_1 = \Gamma_1 \backslash G_1$ and $\mu _ 1 = \pi _*\mu$ (by slight abuse of notation we use $\pi$ to also denote the natural map $X \to X_1$)\footnote{Note that our assumptions imply that $\Gamma _ 1$ is a lattice in $G _ 1$.}. Let $U \leq G_a^{-}$ and $U _ 1 \leq (G_1)_{a _ 1} ^{-}$ be~$a$-normalized (respectively, $a_1$-normalized) Zariski connected unipotent subgroups with $U \leq \pi ^{-1} (U _ 1)$. Then
\begin{equation}\label{relative entropy inequality}
h _ \mu (a, U) \leq h _ {\mu _ 1} (a _ 1, U _ 1) + h _ {\mu} (a, U \cap N)
,\end{equation}
with equality holding for $U=G_a^-$ and $U_1=(G_1)_{a _ 1} ^{-}$.
\end{proposition}

\begin{proof}
Since all quantities appearing in \eqref{relative entropy inequality} for general invariant measures are averages of these quantities over their ergodic components we may without loss of generality assume that $\mu$ (and hence $\mu _ 1$) are ergodic.

Let $\mathcal{P}_0$ and $\mathcal{P} _ 1$ be countable finite entropy generating partitions of $X$ and $X_1$ modulo $\mu$ and $\mu _ 1$ respectively so that $\cA_0 = \bigvee_ {n \geq 0} a ^ {- n}  \mathcal{P}_0$
is subordinate to $G _ a ^ -$ and $\cA_1 = \bigvee_ {n \geq 0} a ^ {- n}  \mathcal{P}_1$ is subordinate to $(G _ 1) _ {a _ 1} ^ -$. Setting $\mathcal{P} = \mathcal{P} _ 0 \vee \pi ^{-1} (\mathcal{P} _ 1)$ we obtain a countable partition of $X$ of finite entropy which is certainly generating and is moreover subordinate to $G _ a ^-$ since for every~$x$
\begin{equation*}
[x]_\cA = [x]_{\cA_0} \cap \pi ^{-1} ([\pi (x)] _ {\cA_1})
\end{equation*}
for $\cA = \bigvee_ {n \geq 0} a ^ {- n}  \mathcal{P}$.
Let $\mathcal{A} _ U$ be the $U$-subordinate countably generated $\sigma$-algebra satisfying $\mathcal{A} _ U = \mathcal{P} \vee a ^{-1} \mathcal{A}_U$ as in \S\ref{Entropy contribution}. We shall also make use of the $\sigma$-algebra $(\mathcal{A} _1)_{U_1}$ on~$X_1$ with $(\mathcal{A}_1) _ {U_1} = \mathcal{P}_1 \vee a ^{-1} (\mathcal{A}_1)_{U_1}$. Note that
\[
\mathcal{A}_U \supseteq \pi^{-1}((\mathcal{A} _1)_{U_1}).
\]
Then
\begin{multline*}
h _ \mu (a, U) = H _ \mu (\mathcal{A} _ U \cond a ^{-1} \mathcal{A} _ U) =
H _ \mu (\mathcal{P} \vee a ^{-1}  \mathcal{A} _ U \cond a ^{-1}  \mathcal{A} _ U) \\ = \frac {1 }{ n} H _ \mu \left(\bigvee_ {k = 0}^{n-1} a^{k}  \mathcal{P} \vee a ^{-1}  \mathcal{A} _ U \cond a ^{-1}  \mathcal{A} _ U\right)
=\frac {1 }{ n} H _ \mu \left(\bigvee_ {k = 0}^{n-1} a^{k}  \mathcal{P} \cond a ^{-1}  \mathcal{A} _ U\right)
.\end{multline*}
We now write
\begin{multline*}
\frac 1n H _ \mu \left(\bigvee_ {k = 0}^{n-1} a^{k}  \mathcal{P} \cond a ^{-1}  \mathcal{A} _ U\right) = \frac 1n H _ \mu \left(\bigvee_ {k = 0}^{n-1} a^{k} \pi ^{-1} (\mathcal{P} _ 1) \cond a ^{-1}  \mathcal{A} _ U\right) + \ \ \  \\
+ \frac 1n H _ \mu \left(\bigvee_ {k = 0}^{n-1} a^{k} \mathcal{P}  \cond \bigvee_ {k = 0}^{n-1} a^{k} \pi ^{-1} (\mathcal{P} _ 1) \vee a ^{-1} \mathcal{A} _ U\right)
.\end{multline*}
 Regarding the first term, we have
\begin{multline}\label{eq:first entropy}
\frac 1n H _ \mu \left(\bigvee_ {k = 0}^{n-1} a^{k} \pi ^{-1} (\mathcal{P} _ 1)\cond a ^{-1}  \mathcal{A} _ U\right)
\leq
\frac 1n H _ \mu \left(\bigvee_ {k = 0}^{n-1} a^{k} \pi ^{-1} (\mathcal{P} _ 1) \cond a ^{-1}  \pi ^{-1} \bigl((\mathcal{A}_1) _ {U_1}\bigr)\right)\\
= \frac 1n H _ {\mu_1} \left(\bigvee_ {k = 0}^{n-1} a_1^{k} \mathcal{P} _ 1 \cond a_1 ^{-1}  (\mathcal{A}_1) _ {U_1}\right)
= h_{\mu_1}(a_1, U_1).
\end{multline}
Regarding the second term, 
\begin{multline*}
\frac 1n H _ \mu \left(\bigvee_ {k = 0}^{n-1} a^{k} \mathcal{P}  \cond \bigvee_ {k = 0}^{n-1} a^{k} \pi ^{-1} (\mathcal{P} _ 1) \vee a ^{-1} \mathcal{A} _ U\right) =\\
\begin{aligned}
\qquad\qquad&=
\frac 1n \sum_{k=0}^{n-1} H _ \mu \left( a^{k} \mathcal{P}  \cond  \bigvee_ {l = k}^{n-1} a^{l} \pi ^{-1} (\mathcal{P} _ 1) \vee a ^{k-1} \mathcal{A} _ U\right)
\\
&= \frac 1n \sum_{k=0}^{n-1} H _ \mu \left(\mathcal{P}  \cond \bigvee_ {l = 0}^{n-1-k} a^{l} \pi ^{-1} (\mathcal{P} _ 1) \vee a ^{-1} \mathcal{A} _ U\right)
\end{aligned}
\end{multline*}
By the martingale convergence theorem for entropy \cite[Ch.~2, Thm.~6]{Parry-topics}, as $n \to \infty$
\begin{align*}
H _ \mu \left(\mathcal{P}  \cond \bigvee_ {l = 0}^{n} a^{l} \pi ^{-1} (\mathcal{P} _ 1) \vee a ^{-1} \mathcal{A} _ U\right) &\to H _ \mu \left(\mathcal{P}  \cond \bigvee_ {l = 0}^{\infty} a^{l} \pi ^{-1} (\mathcal{P} _ 1) \vee a ^{-1} \mathcal{A} _ U\right)\\
&=
H _ \mu \left(\mathcal{P}  \cond \pi ^{-1} \left(\bigvee_ {l = -\infty}^{\infty} a_1^{l} \mathcal{P} _ 1\right) \vee a ^{-1} \mathcal{A} _ U\right).
\end{align*}
Let $\mathcal B_1 = \bigvee_ {l = -\infty}^{\infty} a_1^{l} \mathcal{P} _ 1$ which (since $\cP_1$ is generating) is equivalent mod $\mu_1$ to the Borel $\sigma$-algebra on $X_1$. Then $\pi^{-1} \mathcal B_1 \vee \mathcal{A} _ U$ is an $a$-decreasing $\sigma$-algebra subordinate to $U \cap N$ hence
\begin{align*}
H _ \mu \left(\mathcal{P}  \cond \pi ^{-1} \mathcal B_1 \vee a ^{-1} \mathcal{A} _ U\right) &= H _ \mu \left(\pi^{-1} \mathcal B_1 \vee \mathcal{A} _ U  \cond a^{-1} \left(\pi^{-1} \mathcal B_1 \vee \mathcal{A} _ U\right)\right)\\
&= h_\mu(a, U\cap N)
\end{align*}
hence
\begin{equation}\label{eq:second entropy}
\frac 1n H _ \mu \left(\bigvee_ {k = 0}^{n-1} a^{k} \mathcal{P}  \cond \bigvee_ {k = 0}^{n-1} a^{k} \pi ^{-1} (\mathcal{P} _ 1) \vee a ^{-1} \mathcal{A} _ U\right) \to h_\mu(a, U \cap N) \qquad\text {as $n \to \infty$}
.\end{equation}
The estimates \eqref{eq:first entropy} and \eqref{eq:second entropy} together imply \eqref{relative entropy inequality}.

When $U = G^{-} _ a$ and $U _ 1 = (G _ 1) ^ - _ {a _ 1}$ equation \eqref{relative entropy inequality} (with $=$ instead of $\leq$) is essentially the Abramov-Rohklin conditional entropy formula; since \eqref{relative entropy inequality} has already been established, to establish equality (even without referring to Abramov-Rohklin formula) it  only remains to show that $h _ \mu (a) \geq h _ {\mu _ 1} (a _ 1) + h _ {\mu} (a, G^-_a \cap N)$ which can be established in a very similar way to the above argument starting from the defining formula for entropy
\begin{equation*}
h _ \mu (a) = \lim_ {n \to \infty} \frac 1n H _ \mu \left(\bigvee_ {k = 0} ^ {n - 1} a ^ k \mathcal{P}\right).
\end{equation*}
This and the Abramov-Rohklin conditional entropy formula $h _ \mu (a) = h _ {\mu _ 1} (a _ 1) + h _ {\mu} (a |X_1)$ can be used to show that $h _ {\mu} (a, G^-_a \cap N)=h _ {\mu} (a |X_1)$. Alternatively one could show $h _ {\mu} (a, G^-_a \cap N)=h _ {\mu} (a |X_1)$ directly by employing an appropriate $a$-decreasing $\sigma$-algebra to evaluate the left-hand side and use the Abramov-Rohklin conditional entropy formula to establish equality holds in \eqref{relative entropy inequality} for $U=G_a^-$. We leave the details to the imagination of the reader (or cf.~\cite[Prop.~3.3]{Einsiedler-Lindenstrauss-joining}).
\end{proof}

\begin{proposition}[cf.~{\cite[Coro.~3.4]{Einsiedler-Lindenstrauss-joining}}]\label{equality for coarse Lyapunov proposition}
Let $G, \Gamma, X=\Gamma \backslash G, N, G _ 1, \Gamma _ 1, X_1=\Gamma_1 \backslash G_1$ and $\pi$ be as in Proposition~\ref{entropy inequality proposition}. Let $a < G$ be a class-$\mathcal{A}'$ homomorphism $\Z ^ d \to G$, and set $A = a(\Z^d)$ and $A_1 = \pi \circ a (\Z ^ d)$. Let $\mu$ be an $A$-invariant and ergodic measure on $X$ and $\mu _ 1 = \pi _ {*} \mu$ the corresponding invariant measure on $X _ 1$.
Then for any coarse Lyapunov weight $[\alpha]$
\begin{equation*}
h _ \mu (a, G^{[\alpha]}) = h _ {\mu _ 1} (a _ 1, (G _ 1)^{[\alpha]}) + h _ {\mu} (a, G^{[\alpha]} \cap N)
.\end{equation*}
\end{proposition}

\begin{proof}
Fix $a \in A$ with $G ^ {[\alpha]} < G _ a ^ -$, and let $[\alpha _ 1]=[\alpha], \dots, [\alpha _ k]$ be all the coarse Lyapunov weights contracted by $a$, i.e. so that $G _ a ^ - = \prod_ {i = 1} ^ k G ^ {[\alpha _ i]}$. Then by Proposition~\ref{entropy inequality proposition}, for every $i=1$, \dots, $k$
\begin{equation} \label{equation for each coarse Lyapunov}
h _ \mu (a, G^{[\alpha _ i]}) \leq h _ {\mu _ 1} (a _ 1, (G _ 1)^{[\alpha _ i]}) + h _ {\mu} (a, G^{[\alpha _ i]} \cap N)
\end{equation}
while for $G _ a ^ -$ we have the \emph{equality}
\begin{equation} \label{equation for all of them together}
h _ \mu (a, G^{-}_a) = h _ {\mu _ 1} (a _ 1, (G _ 1)^{-}_{a_1}) + h _ {\mu} (a, G^{-}_a \cap N)
.\end{equation}
However, by \eqref{Sum formula of U entropy} each term of \eqref{equation for all of them together} is the sum of the corresponding term in \eqref{equation for each coarse Lyapunov} over $i = 1, \dots,k$. Thus equality holds in \eqref{equation for each coarse Lyapunov} for all $i$, in particular for $i = 1$.
\end{proof}

Finally, we are in position to arrive at what we sought after: restrictions on the support of $\mu ^ {[\alpha]} _ x$ arising from the fact that $\mu$ is a joining, or more generally that $\pi _ {*} \mu$ is the uniform measure on $X _ 1$.

\begin{corollary} [cf.~{\cite[Coro.~3.5]{Einsiedler-Lindenstrauss-joining}}]\label{large P corollary}
Assume in addition to the assumptions of Proposition~\ref{equality for coarse Lyapunov proposition} that $\pi _*\mu = m _ {X _ 1}$ (i.e.~to the appropriate uniform measure). Let $[\alpha]$ be a course Lyapunov weights, and let $P  ^{[\alpha]}$ be the group $P_U$ given by Lemma~\ref{cor-eigenvalue} for $U= G ^ {[\alpha]}$ and $\mu$. Then $\pi (P  ^{[\alpha]}) = (G _ 1) ^ {[\alpha]}$.
\end{corollary}

\begin{proof}
Suppose in contradiction that $\pi (P ^ {[\alpha]}) \neq (G _ 1) ^ {[\alpha]}$. Then on the one hand by Proposition~\ref{equality for coarse Lyapunov proposition}
\[
h _ \mu (a, G^{[\alpha]}) = h _ {m _ 1} (a _ 1, (G _ 1)^{[\alpha]}) + h _ {\mu} (a, G^{[\alpha]} \cap N)
\]
with $m_1=m_{X_1}$ the uniform measure on $X _ 1$. On the other hand since $\mu _ {x, G^{[\alpha]}}$ is supported on $P ^ {[\alpha]}$ almost everywhere, $h _ \mu (a, G^{[\alpha]})=h _ \mu (a, P^ {[\alpha]})$, whereby applying Proposition~\ref{entropy inequality proposition} we get
\begin{align*}
h _ \mu (a, G^{[\alpha]})&=h _ \mu (a, P^ {[\alpha]}) \leq h _ {m _ 1} (a _ 1, \pi (P^ {[\alpha]})) + h _ {\mu} (a, (NP^ {[\alpha]}) \cap N) \\
& \leq
h _ {m _ 1} (a _ 1, \pi (P^ {[\alpha]})) + h _ {\mu} (a, G^{[\alpha]} \cap N)
.\end{align*}
This gives a contradiction as clearly $h _ {m _ 1} (a _ 1, (G _ 1)^{[\alpha]})> h _ {m _ 1} (a _ 1, \pi (P^ {[\alpha]}))$.
\end{proof}

\section{Classification of ergodic joinings --- proof of Theorem~\ref{higher rank}}\label{proof-higher-rank}

The main step in the proof of Theorem~\ref{higher rank} is the following proposition. Once this proposition is established, Theorem~\ref{higher rank} can be proved by induction on $r$ using the deep classification results for unipotent groups, cf.~\S\ref{unipotent measure rigidity}.

\begin{proposition}\label{unipotent invariance proposition}
Let $\mu$ be an $A$-invariant and ergodic joining on $X = \prod_ i X _ i$ as in Theorem~\ref{higher rank}. Then $\mu$ is invariant under an $A$-normalized Zariski closed and Zariski connected unipotent subgroup $U$.
\end{proposition}

We shall prove this proposition in \S\ref{proof of invariance} below; the inductive argument going from Proposition~\ref{unipotent invariance proposition} to Theorem~\ref{higher rank} will be given in \S\ref{inductive section}.

\subsection{Proof of Proposition~\ref{unipotent invariance proposition} under a mild assumption}\label{proof of invariance}

In \S\ref{fine structure} we considered the groups $P ^ {[\alpha]} \leq G ^ {[\alpha]} $ (with $[\alpha]$ a coarse Lyapunov weight) which were defined as the smallest Zariski connected closed subgroups of $G ^ {[\alpha]}$ containing the support of the leafwise measures $\mu _ x ^ {[\alpha]}$. For a fixed $[\alpha]$, by Lemma~\ref{cor-eigenvalue} the smallest such group for  $\mu _ x ^ {[\alpha]}$ is independent of~$x$ outside a set of $\mu$-measure zero, and this group is normalized by $A$.
These groups would be
key players in the proof of Proposition~\ref{unipotent invariance proposition} and what would be crucial for us is that under the joining assumption we know that they are quite large: specifically, Corollary~\ref{large P corollary} implies that if $\pi _ i$ denotes the natural projection $G \to G _ i $ (in the notations of Theorem~\ref{higher rank})
\begin{equation}\label{big projection reference}
\pi_i (P ^ {[\alpha]}) =  G_i ^{[\alpha]}
\end{equation}
for $i=1, \dots, r$.

To better present the crux of the matter, we initially prove Proposition~\ref{unipotent invariance proposition} under the additional assumption:
\begin{equation}\label{assumption for proposition}
\text {\emph{there is some $[\alpha _ 0]\in[\Phi]$ and $a \in A$ with $\alpha _ 0 (a) = 0$}}
\end{equation}
where we recall that $[\Phi]$ denotes the set of all (non-zero) coarse Lyapunov weights.

The general case can be reduced to this using a fairly straightforward averaging argument similar to the one employed in \S\ref{construction}; we present this reduction in \S\ref{removing an assumption} below.
Set~$P^-=\prod_{[\alpha]:\alpha(a)<0}P^{[\alpha]}$. By Theorem~\ref{EKtheorem}, the group~$P^-$ is for a.e.~$x$ the smallest
Zariski connected unipotent subgroup of~${G}_a^-$ that contains~$\supp\mu_x^{{G}_a^-}$. This implies that~$\mu_x^{P^-}=\mu_x^{{G}_a^-}$ a.e., or more precisely that for any Borel subset $F \subset {G}_a^-$
\[
\mu_x^{P^-} (F \cap P^-) =\mu_x^{{G}_a^-}(F).
\]
Moreover, by definition, the measure~$\mu$ is faithfully~$P^-$-recurrent. Define similarly $P^+=\prod_{[\alpha]:\alpha(a)>0}P^{[\alpha]}$.

There are now two cases to consider: the ``high entropy'' case, when $P ^ {[\alpha _ 0]}$ does not commute with at least one of the groups $P ^ -$ and~$P^+$, and the ``low entropy'' case when they do.

In the ``high entropy'' case, we find a $a ' \in A$ which contracts both $G ^ {[\alpha _ 0]}$ and $P ^ -$, and a direct application of Theorem~\ref{high entropy} shows that $\mu$ is invariant under $[P ^ {[\alpha _ 0]}, P ^ -]$ --- a Zariski close, Zariski connected $A$-normalized unipotent group, establishing the proposition.

Hence \emph{we need only consider the case when both~$P ^ -$ and~$P^+$ commute with~$P ^ {[\alpha _ 0]}$}. We would like to employ Theorem~\ref{low entropy theorem} for $\mu$ with respect to the already chosen $a \in A$ and with $P^-$ (and similarly also $P^+$) playing the role of $U$.

We need first to verify that $P^-$ satifies the conditions imposed on by Theorem~\ref{low entropy theorem}, namely that for all nontrivial $g \in G ^ +_a$ there exists some $ u \in P^-$ with $u g u ^{- 1} \not \in G ^ + _aG ^ 0_a$.
In order to show this we use the following fact about orbits of unstable horospheric groups in algebraic representations, which will also be used later in the paper, and which we state in somewhat greater generality than is needed for the purposes of this proof.

\begin{lemma}\label{shearing}
 Suppose $\G$ is a Zariski connected algebraic group defined over~$\Q_\sigma$
 whose radical equals its unipotent radical, and $a \in \G(\Q_\sigma)$ an element of class-$\mathcal A'$. 
 Let~$\rho:\G\to\operatorname{Aut}(V)$ be an algebraic representation of $G$ on a $\Q_\sigma$-vector space. Let $V^0_a$ be the eigenspace for eigenvalue 1 of~$\rho(a)$, and $V_a^-$ and~$V_a^-$ the sum of the eigenspaces of $\rho(a)$ with eigenvalues of absolute value $<1$ (resp.\ the sum of eigenspaces of $\rho(a)$ with eigenvalues of absolute value $>1$). Then for any $w \in V_a^+$ there is an element  $u \in G_a^-$ so that  $\rho(u).w\notin  V_a^++V^0_a$. In particular,
for every nonzero $w\in\mathfrak g_a^+$ there exists some 
 $u\in G_a^-$ with $\Ad_u(w)\notin \mathfrak g_a^++\mathfrak g^0$.
\end{lemma}

\noindent
We note that the assumption of the lemma holds if~$\G$ is e.g.\ semisimple or perfect. 
This and related properties of algebraic representations have already played an important role in homogeneous dynamics in the context of unipotent flows, see e.g.\ \cite[pp. 105--125]{Shah-expanding-translates}.

\begin{proof}[Proof of Lemma~\ref{shearing}]
The second part of the lemma follows from the first by taking $\rho=\operatorname{Ad}$ and~$V=\mathfrak g$.
Let now~$\rho:\G\to\operatorname{Aut}(V)$ be an algebraic representation, and define~$V^0_a$, $V_a^-$, $V_a^+$ as above.
If the conclusion would have been false, then $\rho(G_a^-)(w)\subset  V_a^++V^0_a$
and hence 
	\[
	\rho(G_a^+G_a^0)\rho(G_a^-)(w)\subset V_a^++V_a^0.
	\]
However, $G_a^+G_a^0G_a^-$
is Zariski-dense in $G$ and so $\rho(G_a^-)(w)\subset V_a^++V^0$ would generated
a subrepresentation containing~$w$, and since $w \in V_a^+$ the determinant of $\rho$ restricted to this subrepresentation would give a non-trivial
algebraic character of $\G$. This is a contradiction to the assumption on the radical of~$\G$ which implies
	that~$\G$ cannot have any nontrivial algebraic characters.
\end{proof}

Now let~$g=\exp(w)\in {G}_a^+$ be a nontrivial element.
By the above~$\pi_i(P^-)=(G_i)_{a_i}^-$ and similarly~$\pi_i(P^+)=(G_i)_{a_i}^+$
 for~$i=1,\ldots,r$, where $a_i=\pi_i(a)$.
As~$g$ is nontrivial, there exists some~$i$ such that~$\pi_i(g)$ is nontrivial.
This implies that there exists some~$u_i\in (G_i)_{a_i}^-$ 
with~$u_i\pi_i(g)u_i^{-1}\notin (G_i)_{a_i}^+(G_i)_{a_i}^0$ by Lemma~\ref{shearing}, 
which implies 
that there exists some~$u\in P^-$ with~$ugu^{-1}\notin {G}_{a}^+{G}_{a}^0$ 
as required. 

Thus Theorem~\ref{low entropy theorem} is applicable for $\mu$, $a$ and $P^-$, and for similar reasons it is also applicable for~$\mu$ and~$a$ if we use $P^+$ in the role of $U$.
We may conclude that for both $P^-$ and $P^+$, the measure $\mu$ must satisfy either \ref{not transient item} or \ref{aligned item} of that theorem.
It would transpire that \ref{not transient item} of Theorem~\ref{low entropy theorem} for either $U=P^-$ or $U=P^+$ leads to unipotent invariance, whereas \ref{aligned item} of the same theorem holding simultaneously for both these choices for $U$ would lead to a contradiction.

\subsubsection{Assume \ref{not transient item} of Theorem~\ref{low entropy theorem} holds for at least one of the choices~$U=P^{\pm}$}\label{not transient holds}

The argument in this case is essentially identical to that in \cite[Prop.~6.2]{full-torus-paper} or \cite[Lemma~8.4]{Low-entropy}; we provide it for the sake of completeness.

Assume e.g.~\ref{not transient item} holds for $U=P^-$. Recall that this means that~$\mu$ is not~$C_{{G}}(U)\cap {G}_a^-$-transient, i.e. that we can find for any $\mu$-null set $Y_0 \subset X$ two points~$x,u.x \not \in Y_0$
with some nontrivial~$u \in C_{{G}}(P^-)\cap {G}_a^-$ such that~$\mu_{x,P^-}=\mu_{u.x,P^-}$; we shall assume as we may that $Y_0$ contains the complement of the co-null set $X'$ of~\S\ref{basic properties}.
In view of \ref{5.1.3} of \S\ref{basic properties}, as $\supp \mu_{x,G_a^-} \subseteq P^-$ a.s., for any $u \in {G}_a^-$ if both points~$x,u.x$ are outside $Y_0$ the element $u$ satisfies~$u \in P^-$.
It follows that there is a nontrivial $u \in P^-$ and a $x \not \in Y_0$ so
that $u.x \not \in Y_0$ and ~$\mu_{u.x,P^-}=\mu_{x,P^-}$. Applying once more~\ref{5.1.3} of~\S\ref{basic properties} we see that
\(
\mu _ {x, P ^ -} \propto (\mu _ {u . x, P ^ -})u
\)
hence
\[
\mu _ {x, P ^ -} \propto (\mu _ {x, P ^ -})u
.\]
Since by \eqref{volume decay} and \eqref{module bound} (or alternatively by the more general growth estimate~\cite[Thm. 6.29]{Pisa-notes}) the function
\[
R \mapsto \mu _ {x, P ^ -}(B^{P^-}_R)
\]
grows at most polynomially, we get in fact that $\mu _ {x, P ^ -} = (\mu _ {x, P ^ -})u$.

From this, using Poincare recurrence and ergodicity, it is easy to conclude that $\mu$ is invariant under a $A$-normalized Zariski closed and Zariski connected unipotent subgroup $U \leq P ^ -$ as in e.g.~\cite[\S5--6]{KatokSpatzier96} or \cite[Prop.~6.2]{full-torus-paper}.
Indeed, for every $x$ denote 
by $H_x$ the (closed) group
\[
H _ x  = \left\{ u \in P ^ -:
\mu _ {x, P ^ -} = (\mu _ {x, P ^ -})u \right\};
\]
by the above discussion we note that for a set of positive $\mu$-measure of $x \in X$ the group $H _ x$ is nontrivial.
By \ref{5.1.4} of \S\ref{basic properties} it follows that for any $a ' \in A$ we have that
$H_{a' . x} = a' H_{x}(a')^{-1}$ a.s. In particular, by ergodicity of $A$ this implies that $H_x$ is nontrivial a.s. By Poincare recurrence for $a$ which contracts $P^-$ we have that a.s.~$H_x$ contains arbitrarily small and arbitrary large elements of $P ^ -$. Since $H _ x$ is a closed subgroup of $P ^ -$, a product of unipotent groups over local fields of characteristic zero, we see that $H_x$ contains a Zariski closed and Zariski connected unipotent subgroup; let  $V_x$ be the largest such subgroup of $H_x$. 
Using Poincare recurrence, ergodicity under $A$, and the assumptions that $A$ is of class-$\mathcal A'$, we may conclude as in Proposition~\ref{cor-eigenvalue} that $V_x$ is a.s. equal to a single $A$-normalized unipotent group $V$, and hence we have established Proposition~\ref{unipotent invariance proposition} under the assumption highlighted above.

\noindent
The argument for~$P^+$ is the same.

\subsubsection{Assume~\ref{aligned item} holds in Theorem~\ref{low entropy theorem} for both $U=P^+$ and $U=P^-$}\label{aligned leads to ruin}
In other words, we are assuming that~$\mu$ is both locally~$(\mu_{x,P^-},C_G(P^-))$-aligned and locally~$(\mu_{x,P^+},C_G(P^+))$-aligned (formally, \ref{aligned item} of Theorem~\ref{low entropy theorem} gives a bit more information, but this will suffice to arrive at a contradiction).

To see that our assumption leads to contradiction we shall use the recurrence of $ \mu$ under the subgroup~$P^{[\alpha_0]}$, which we know is nontrivial by Corollary~\ref{large P corollary}. This deviates from the strategy in previous applications of the low entropy method, e.g. in~\cite{full-torus-paper} where the torus action was used for this purpose. The reason this action is relevant to us is that since $P^{[\alpha_0]}$ commutes with both $P^+$ and $P^-$ (otherwise we would have already gotten unipotent invariance), by \eqref{mu product corollary}
we have that on a set of full measure $X'$, if both $x$ and $u.x \in X'$ for $u \in P^{[\alpha_0]}$ then
\begin{equation}\label{same conditional measures}
\mu_{x,P^-}=\mu_{u.x,P^-} \quad\text{and}\quad \mu_{x,P^+}=\mu_{u.x,P^+}
\end{equation}

The alignment property implies that there is a subset $B \subset X$ with measure $\mu (B) > 0.9$ and a $\delta>0$ so that if $x \in B$ and $y \in B \cap B ^ X _ \delta (x)$ have the same $P ^ -$-leafwise measures, i.e. $\mu _ {x, P ^ -} = \mu _ {y, P ^ -}$, then $x = h.y$ for $h$ in the unit ball in $C _ G (P ^ -)$.
Similarly, whenever $x \in B$ and $y \in B \cap B ^ X _ \delta (x)$ have the same $P^+$-leafwise measures then $x = h.y$ for $h$ in  the unit ball of $C _ G (P ^ +)$. We can assume $B \subset X '$, hence if we have $x \in B$ and (a possibly big) $u \in P ^ {[\alpha _ 0]}$ so that $u . x \in B \cap B _ \delta ^ X (x)$
then by \eqref{same conditional measures} the point $u . x$ is in $O . x$ for $O$ a fixed compact neighborhood of the identity in $C _ G (P ^ -) \cap C _ G (P ^ +) = C _ G (\langle P ^ -, P ^ + \rangle)$. In other words, writing $x = \Gamma g$ with $g \in G$, we have established that under the above assumptions
\begin{equation*}
\gamma g u ^{-1} \in gO \qquad\text{for some $\gamma \in \Gamma$}
.\end{equation*}

Fix an index $1 \leq i \leq r$ for which $G _ i ^ {[\alpha _ 0]}$ is nontrivial. It follows from Corollary~\ref{large P corollary} that $\pi _ i (P ^ {[\alpha _ 0]}) = G _ i ^ {[\alpha _ 0]}$, and hence by Proposition~\ref{prop-faithful} for a.e.~$x \in B$ there is a sequence $u _ k \in P ^ {[\alpha _ 0]}$ with $\pi _ i (u _ k) \to \infty$ so that $u _ k . x \in B$ for every $k$ and $u _ k . x \to x$ as $k \to \infty$.
By the alignment property we may conclude that there exists $\gamma _ k \in \Gamma$ and $o_k \in O$ so that $\gamma _ k g u _ k ^{-1} = g o_k$, or in other words
\[
\gamma _ k = g o_k u_k g^{-1}
.\]
Since $O$ is compact and $\pi _ i (u _ k) \to \infty$, the sequence of elements $\gamma _ k ' = \pi _ i (\gamma _ k)$ in $\Gamma _ i$ tend to infinity, and since the group $P ^ {[\alpha _ 0]}$ commutes with both $P ^ -$ and $P ^ +$, 
\begin{equation*}
\gamma _ k ' \in \pi _ i (g C _ G (\langle P ^ -, P ^ + \rangle)g^{-1}) = \pi_i(g) C _ {G _ i} (\langle (G_ i)_{a_i} ^ -, (G_ i)_{a_i} ^ + \rangle)\pi_i(g)^{-1}
.\end{equation*}
Note that by the assumptions on $A$ in Theorem~\ref{higher rank} the element $a _ i$ is nontrivial, and since it is of class-$\cA'$ both $(G_i) ^ - _ {a _ i}$ and $(G _ i) ^ + _ {a _ i}$ are nontrivial.
Let~$\fl_i$ be the Lie algebra generated by the Lie algebras of~$(G _ i) ^ - _ {a _ i}$ and~$(G _ i) ^ + _ {a _ i}$. This algebra
 is in fact an ideal in~$\fg_i$ and is called the Auslander ideal (that $\fl_i$ is an ideal can be seen using~\eqref{weights add} and the Jacobi identity). 
 
 Since $\GG_i$ is $\Q$-almost simple, the group~$G_i$ has finite center, hence $\gamma'_k \to \infty$ is incompatible with the folowing lemma, giving the sought after contradiction, completing the proof of Proposition~\ref{unipotent invariance proposition} under the assumption \eqref{assumption for proposition}.

 \begin{lemma}\label{simple one}
	 Assume~$\gamma \in  C_{G_i}(\fl_i) \cap \Gamma_i$. Then~$\gamma$
	belongs to the center of~$G_i$. 
\end{lemma}

\begin{proof}
  Note that~$C_{\GG_i}(\gamma)$ is an algebraic group defined over~$\Q$ and the group of rational elements
  of the Zariski connected group~$\GG_i$ are Zariski dense, hence
  \[
    \MM=\bigcap_{\eta\in\GG_i(\Q)}\eta\, C_{\GG_i}(\gamma)\,\eta^{-1}=\bigcap_{\eta\in{\GG_i}(\Q)} C_{\GG_i}(\eta\gamma\eta^{-1})
    \unlhd\GG_i
  \]
 is a normal subgroup of~$\GG_i$ defined over~$\Q$. Since the Lie algebra of $\MM$ contains $\fl_i$, it is of positive dimension. Since $\GG _ i$ is Zariski connected and $\Q$-almost simple it follows that $\MM=\GG _i$. Thus any $h \in G_i$ commutes with $\gamma$ as claimed .
\end{proof}

\subsection {Proof of Proposition~\ref{unipotent invariance proposition} with no assumptions}
\label{removing an assumption}
In this section we explain how Proposition~\ref{unipotent invariance proposition} can be deduced from the special case of this proposition where \eqref{assumption for proposition} is assumed. We remark that this assumption is often automatically satisfied: specifically, if $S$ contains at least one finite place $p$, and if for at least one $i = 1, \dots, r$ we have that $A _ i$ has a nontrivial projection to $\GG_i (\Q _ p)$, then for any coarse Lyapunov weight $[\alpha _ 0]$ appearing in the Lie algebra of $\GG_i (\Q _ p)$ there is an element $a \in A$ satisfying \eqref{assumption for proposition}.

Thus the only case we need to consider is that $A<\GG(\R)$.
In this case
we can use the eigenvalue assumption at~$\infty$ and consider~$A\simeq\Zd$ as a lattice
in a connected~$\R$-diagonalizable Lie subgroup~$\bar{A}\simeq\R^d$ 
and study a new measure~$\bar{\mu}$ instead of~$\mu$. Here we define
\[
 \bar{\mu}=\int_{\bar{A}/A}\bar{a}_*\mu\  d\bar{a},
\] 
where we integrate with respect to a normalized Haar measure on the compact abelian 
group~$\bar{A}/A$. Note that~$\bar{\mu}$
is then an ergodic joining for the actions of subgroups~$\R^d\simeq\bar{A_i}<G_i\cap\GG(\R)$. 
Now consider again some weight~$\alpha_0$ and some nontrivial~$\bar{a}\in\ker\alpha_0$ which now definitely exists.
The remainder of the argument is as before and leads to invariance of~$\bar{\mu}$ under a Zariski connected unipotent subgroup $U$ normalized by $\bar A$ (equivalently, is normalized by $A$).

It follows from the Mautner phenomenon that almost every $A$-ergodic component $\mu '$ of $\bar \mu$ is $U$-invariant. But since $\mu$ is $A$-ergodic, every $A$-ergodic component of $\bar \mu$ is of the form $\mu ' = a. \mu$ for some $a \in \bar A$. As $U$ is normalized by $A$ it follows that $\mu$ itself is $U$-invariant.

\subsection{Inductive proof of Theorem~\ref{higher rank} using Proposition~\ref{unipotent invariance proposition}}\label{inductive section}

Let $\mu$ be as in Theorem~\ref{higher rank}.
By Proposition~\ref{unipotent invariance proposition}, the measure~$\mu$ is invariant
under an~$A$-normalized Zariski closed and Zariski connected unipotent subgroup~$U$. 
We now apply Theorem~\ref{MTclassA}
to~$\mu$ and the subgroup~$\langle A, U\rangle$.
It follows that there exists a~$\Q$-group~$\HH$, some~$g\in G$, and a finite index subgroup~$H<\HH(\Q_S)$ containing~$U$
such that~$\mu$ is invariant under~$g^{-1}H g$ and supported on~$\Gamma N^1_G(\HH) g$ for some~$g\in G$ with~$\Gamma g\in\supp\mu$. 
Replacing~$\mu$ by the pushforward by~$g^{-1}$
and also~$A$ by its conjugate by~$g$, we may assume without loss of generality that~$g$ is the identity element~$e$.

The next lemma is phrased more generally than we need it here
as it will also be useful in the proof of Theorem~\ref{thm:unipotent}. 

\begin{lemma}\label{joining-groups}
	Suppose~$\GG=(\GG_1\times\cdots\times\GG_r)\ltimes\UU_{\textrm{rad}}$ is a semidirect product
	of the direct product of~$\Q$-almost simple groups~$\GG_1,\ldots,\GG_r$ and a unipotent radical~$\UU_{\textrm{rad}}$. 
	Suppose~$\MM<\GG$ is a connected subgroup with~$\pi_i(\MM)=\GG_i$ for all~$i=1,\ldots,r$, where~$\pi_i:\GG\to\GG_i$
	is the canonical projection map. Then~$\MM$ can be written as $(\GG_1'\times\cdots\times\GG_{t}')\ltimes\UU'_{\textrm{rad}}$. 
	More precisely, the radical of~$\MM$ equals its unipotent radical and is given 
	by~$\UU'_{\textrm{rad}}=\MM\cap\UU_{\textrm{rad}}$, and the Levy subgroup
	is a direct product of~$\Q$-almost simple groups~$\GG_j'$ for~$j=1,\ldots,t$. After possibly permuting the 
	groups~$\GG_1\times\cdots\times\GG_r$ we have that~$\GG_1'$ is isomorphic
	to a subgroup $\bar{\GG}_1<\GG_1\times\cdots\times\GG_{s_1}$ (where the isomorphism 
	is given by taking the quotient modulo the unipotent radical), is isogeneous to each of
	the groups~$\GG_1,\ldots,\GG_{s_1}$, and that~$\bar{\GG}_1$ projects onto~$\GG_1,\ldots,\GG_{s_1}$ (in the category
	of algebraic groups). Similarly, for~$\GG_2'\cong\bar{\GG}_2<\GG_{s_1+1}\times\cdots\times\GG_{s_2}$, etc.~up
	to~$\GG_{t}'\cong\bar{\GG}_{t}<\GG_{s_{t-1}+1}\times\cdots\times\GG_{r}$.
\end{lemma}

\begin{proof}
	We begin with the case of~$\UU_{\textrm{rad}}$ being trivial, which we will prove by induction on~$r$. 
	Projecting~$\MM<\GG_1\times\cdots\times\GG_r$ to~$\tilde{\GG}=\GG_1\times\cdots\times\GG_{r-1}$ we obtain
  a connected subgroup~$\tilde{\MM}$, which by induction is of the  
  form~$\tilde{\MM}=\tilde\GG_1\times\cdots\times\tilde\GG_{\tilde{r}}$ 
  as in the lemma.  

Notice that~$\pi_r(\MM)=\GG_r$ implies that~$\MM\cap\GG_{r}$ is a normal~$\Q$-subgroup.
If~$\MM\cap\GG_{r}$ has positive dimension, then~$\MM\cap\GG_r=\GG_r$ as~$\GG_r$ is a~$\Q$-almost simple group. However, 
this implies that~$\MM=\tilde{\MM}\times\GG_r$ and hence has the desired form.
So suppose~$\MM\cap\GG_r$ is finite, and due to~$\pi_r(\MM)=\GG _r$ also central. Taking the quotient of~$\GG_r$
by~$\MM\cap\GG_r$ we may consider~$\MM$ as the graph of a homomorphism~$\phi$ from~$\tilde{\MM}$ to~$\GG_r/\MM\cap\GG_r$. As~$\GG_r$ is~$\Q$-almost simple and~$\tilde{\MM}=\tilde\GG_1\times\cdots\times\tilde\GG_{\tilde{r}}$ is defined over~$\Q$ and semi-simple, it follows that the kernel
of this homomorphism must equal the direct product of all but one of the direct factors of~$\tilde{\MM}$. We may
assume that the kernel equals~$\tilde\GG_1\times\cdots\times\tilde\GG_{\tilde{r}-1}$ and see that~$\MM=\tilde\GG_1\times\cdots\times\tilde\GG_{\tilde{r}-1}\times\FF$ for some subgroup~$\FF$ of~$\tilde\GG_{\tilde{r}}\times\GG_r$. It follows that~$\FF$ represents an isogeny between~$\tilde\GG_{\tilde{r}}$
and~$\GG_r$, which once more shows that~$\MM$ has the desired form.

Let us return to the general case of~$\MM<\GG=(\GG_1\times\cdots\times\GG_r)\ltimes\UU_{\textrm{rad}}$ as in the lemma.
Let~$\pi:\GG\to\GG_1\times\cdots\times\GG_r$ be the canonical factor map modulo the unipotent radical. By the above~$\pi(\MM)\cong\GG_1'\times\cdots\times\GG'_{t}$ and each of the~$\Q$-almost simple groups~$\GG_i'$ can be described
as in the lemma. In particular,~$\pi(\MM)$ is semisimple. Let now~$\MM=\LL'\cdot\UU'_{\textrm{rad}}$ 
be a Levy decomposition of~$\MM$ for a reductive~$\Q$-group~$\LL$ and a unipotent~$\Q$-group~$\UU'_{\textrm{rad}}$. 
Clearly~$\UU_{\textrm{rad}}\cap\MM<\UU'_{\textrm{rad}}$, which implies that~$\LL'\cong\pi(\LL')$ and~$\pi(\MM)=\pi(\LL')\pi(\UU'_{\textrm{rad}})$ is a Levy decomposition of~$\pi(\MM)$. Since~$\pi(\MM)$ is semisimple,
we see that~$\LL'\cong\pi(\MM)$ is semisimple and~$\UU_{\textrm{rad}}\cap\MM=\UU'_{\textrm{rad}}$ which gives the lemma.
\end{proof}

The assumption that~$\mu$ is a joining implies in particular that~$\pi _ i (\supp \mu)=X_i$.
As~$\supp \mu \subset \Gamma N^1_G(\HH) g$
this shows that the $\Q$-group $\MM=N_\GG(\HH)^\circ$ satisfies
\begin{equation*}
\pi_i(N_\GG(\HH)^\circ)=\GG_i \qquad\text{for~$i=1,\ldots,r$}.
\end{equation*}
Applying Lemma~\ref{joining-groups} to~$\MM$ we have that~$\MM=N_\GG(\HH)^\circ$
is a direct product of~$\Q$-almost simple algebraic groups~$\GG_1',\ldots,\GG_t'$, where each factor corresponds
possibly to a subset of factors of~$\GG$. Moreover, since~$\HH\unlhd\MM$ we see that~$\HH$ equals
the product of some of the direct factors of~$\MM$, say~$\HH=\GG_{s+1}'\times\cdots\times\GG_t'$ for~$s<t$. 
The definition of~$\MM$ then implies that the other factors~$\GG_1'=\GG_1,\ldots,\GG_s'=\GG_s$ of~$\MM$
actually equal direct factors of~$\GG$. 

We define~$M=G\cap N_\GG(\HH)^\circ(\Q_S)$. Note that as $\MM$ has no $\Q$-characters, $\Gamma M$ is closed in $X$.
Since $A$ is of class-$\cA'$ the Zariski closure of~$A$ is connected, hence as~$A<N_\GG(\HH)(\Q_S)$ it follows that in fact~$A<M$.  Let us also write~$G'=G_1\times\cdots\times G_s$ and~$H_G=G\cap\HH(\Q_S)$ so that~$M=G'\times H_G$
by the above description of~$\MM$ and~$\HH\unlhd\MM$.
Also since~$\mu$ is supported on~$\Gamma N_G(H)$ by Theorem~\ref{MTclassA}
we obtain from ergodicity that~$\mu$ is supported on the closed orbit~$\Gamma M$ (recall that we assume~$\Gamma\in\supp\mu$). 
To summarize, we may consider~$\mu$ as a measure on~$X=X'\times X_\HH$ where~$X'=\Gamma'\backslash G'$ and~$X_\HH=\Gamma_\HH\backslash H_G$ for~$\Gamma'=\Gamma\cap G'$ and~$\Gamma_\HH=\Gamma\cap H_G$. Let $a(n)_\HH$ denote the $H_G$ component of $a(n)$, and let \[V=\{n\in\Z^d:a(n)_\HH\in H\}<\Z^d.\] 
This group has finite index in $\Z^d$, and hence~$A_V=\{a(n):n\in V\}$ is of finite index in $A$. 

Now recall that~$\mu$ is invariant under a finite index subgroup~$H<H_G<\HH(\Q_S)$ from Theorem~\ref{MTclassA}.
Consider the ergodic decomposition of~$\mu$ with respect to the action of~$A_VH$.
Since~$\mu$ is invariant and ergodic under~$AH$, and~$A_VH$ has finite index in~$AH$, it follows that~$\mu$
is actually a finite convex combination of~$A_VH$-invariant and ergodic measures.
We let~$\rho$ be one of these ergodic components of~$\mu$, so that
\[
\mu=\frac{1}{[\Z^d:V]}\sum_{a\in A/A_V}a.\rho.
\] 
We may also suppose that~$\Gamma e\in\supp\rho$.

Let $\bar{\mathcal B}_{X'}$ denote the inverse image of the Borel $\sigma$ algebra on $X'$ under the natural projection~$X\to X'$.
The fiber measures~$\rho_x^{\bar{\mathcal{B}}_{X'}}$ have the form $\delta_{x'}\times\nu_{x}$ with $\nu _x$ a probability measure on $X_\HH$ where $x'$ is the projection of $x$ to $X'$.
Since $\rho$ is $H$ invariant and $H$ acts trivially on $X'$ the measures $\nu _x$ are $H$-invariant a.e. 
By the equivariance property~$a.\mu_x^{\bar{\mathcal{B}}_{X'}}=\mu_{a.x}^{\bar{\mathcal{B}}_{X'}}$ for every~$a\in A_V$. 
Since~$\nu_x$ is invariant under the image $a_\HH$ of~$a$ in~$H_G$, 
the probability measures~$\nu_{x}$ on~$X_\HH$ are constant on the orbits of~$a$. 
By ergodicity we obtain that~$\nu_{x}=\rho_\HH$
for some~$H$-invariant probability measure~$\rho_\HH$ 
on~$X_\HH$.

Clearly~$\rho_\HH$ equals the push-forward of~$\rho$ to~$X_\HH$,
is~$H$-invariant and moreover is $H$-ergodic since $\rho$ is $A_VH$-ergodic and the projection of $A_V$ to $\HH$ is contained in $H$.
As~$H$ is a finite index subgroup of~$H_G$,
this implies that~$\rho_\HH$
is the Haar measure $m_{\Gamma_\HH H}$ on the~$H$-orbit~$\Gamma_\HH H$ (recall that we set things up so that~$\Gamma e\in\supp\rho$). 

Denote the projection of~$A_V$ to~$G'$ by~$A'$.
It follows that~$\rho=\rho'\times m_{\Gamma_\HH H}$
for an~$A'$-ergodic joining~$\rho'$ on~$X'$. By induction we may already suppose that~$\rho'= m_{\Gamma'H'}$
for some finite index subgroup~$H'<\HH'(\Q_S)$ of a semisimple algebraic~$\Q$-group~$\HH'$ that contains~$A'$.
In other words, we have shown that~$\rho=m_{\Gamma (H'\times H)}$. 

As $A$, and hence the projection of $A$ to $\GG'$, is of class-$\mathcal A'$, the Zariski closure of $A'$ contains the projection of $A$ to $\GG'$, hence
this projection too belongs to~$\HH'$. 
As a finite index subgroup of a semi-simple
group is always normalized by all elements of any non-compact factor, it follows that~$A$ normalizes~$H'\times H$.
In particular we see that~$H''=(H'\times H) A$ is still a finite index subgroup of
the group of~$\Q_S$-points of~$\HH''=\HH'\times\HH$, which gives
\[
 \mu=\frac{1}{[\Z^d:V]}\sum_{a \in A/A_V}a.(m_{\Gamma H'\times H})=m_{\Gamma H''}
\]
and so the theorem.

\subsection{Proof of Corollary~\ref{cor:pairs}}
By Theorem~\ref{higher rank} the joining~$\mu$ is algebraic, so there exists a~$\Q$-subgroup~$\HH<\GG$ such that~$\mu$
is invariant under~$g^{-1}H g$ and supported on~$\Gamma Hg$ for some finite index subgroup~$H=\HH(\Q_S)$ and some~$g\in G$. 
Since~$\mu$ is a joining,~$\HH$ has to satisfy the assumption of Lemma~\ref{joining-groups}. Together with the 
assumption of the corollary the lemma implies that~$\HH=\GG_1\times\cdots\times\GG_r$ and the corollary follows.

\section{The proof of Theorem~\ref{thm:unipotent}}
\label{proof-unipotent}
The proof follows the same outline as the proof of Theorem~\ref{higher rank} in Section~\ref{proof-higher-rank},
we will be brief where there is little difference in the argument and provide details at the steps that are different.

\subsection{The assumptions for Theorem~\ref{low entropy theorem} and unipotent invariance}
Recall that in Theorem~\ref{thm:unipotent} the measure~$\mu$ is an~$A$-invariant
and ergodic probability measure such that~$\pi_*(\mu)$ equals the Haar 
measure on~$(\Gamma_1\times\cdots\times\Gamma_r)\backslash L$,
where~$\pi:G\to L$ is the projection map modulo the unipotent radical. 
As in \S\ref{proof-higher-rank}, we first prove the following important step towards Theorem~\ref{thm:unipotent}:

\begin{proposition}\label{unipotent invariance proposition 2}
Let $\mu$ be an $A$-invariant and ergodic probability measure as in Theorem~\ref{thm:unipotent}. Then $\mu$ is invariant under an $A$-normalized Zariski closed and Zariski connected unipotent subgroup $U$.
\end{proposition}

Assume first that 
\[
\textit{there is some~$[\alpha_0]\in[\Phi]$ and $a\in A$ with~$L^{[\alpha_0]}\neq \{e\}$ and ~$\alpha_0(a)=0$.}
\]
The general case follows from this as in \S\ref{removing an assumption}; we omit the details of this easy reduction.
 
By Corollary~\ref{large P corollary},
it follows that if we let~$P^{[\alpha]}$ be the subgroup from Lemma~\ref{cor-eigenvalue}
we see that~$\pi(P^{[\alpha]})=L^{[\alpha]}$ for any coarse Lyapunov weight~$[\alpha]$. 
 Set
\[
P^-=\prod_{[\alpha]:\alpha(a)<0}P^{[\alpha]};
\]
by Theorem~\ref{EKtheorem} this group is for a.e.~$x$ the smallest
Zariski connected unipotent subgroup of~${G}_a^-$ that contains~$\supp\mu_x^{{G}_a^-}$. This implies that~$\mu_x^{P^-}(C \cap P^-)=\mu_x^{{G}_a^-}(C)$ for all $C\subset G_a^-$ a.e.\
and that~$\mu$ is faithfully~$P^-$-recurrent (cf.~Remark~\ref{faithful remark}). As in \S\ref{proof of invariance}, we may assume~$P^{[\alpha_0]}$ commutes with both~$P^-$ and~$P^+$, as otherwise we may invoke Theorem~\ref{high entropy}
and obtain unipotent invariance of~$\mu$.

It remains to check that $P^-$ satisfies the conditions set on $U$ by
Theorem~\ref{low entropy theorem}, namely that for all nontrivial $g =\exp(w)\in G ^ +_a$ there exists some $ u \in P^-$ with $u g u ^{- 1} \not \in G ^ + _aG ^ 0_a$.
By the above~$\pi(P^-)=L_{a}^-$ and similarly~$\pi(P^+)=L_{a}^+$. 
If~$\pi(g)$ is nontrivial, we have  
that~$u'\pi(g)(u')^{-1}\notin L_{a}^+\,L_{a}^0$ 
for some~$u'\in L_{a}^-$ by Lemma~\ref{shearing},  which implies 
that there exists some~$u\in P^-$ with~$ugu^{-1}\notin {G}_{a}^+{G}_{a}^0$ 
as required.  

So assume now that~$g=\exp(w)\in U_{\textrm{rad}}$. Let~$\mathfrak u_0$ be the Lie algebra
of~$U_{\textrm{rad}}$ and define inductively~$\mathfrak u_{k+1}=[\mathfrak u_0,\mathfrak u_k]$
for~$k\geq 0$. Since~$w\neq 0$ there exists a maximal~$k$ with~$w\in\mathfrak u_k\setminus \mathfrak u_{k+1}$.
Consider now the representation~$\rho$ of~$\G$ on~$V=\mathfrak u_k/\mathfrak u_{k+1}$
induced by the adjoint representation.
We note that by the choice of this representation we have that~$\UU_{\textrm{rad}}$ acts trivially
on~$V$. Applying Lemma~\ref{shearing} gives that there exists some~$u'\in L_a^-$
with~$\rho(u')(w)\notin V_a^++V_a^0$. However, since~$\pi(P^-)=L_a^-$
this implies there must be an element~$u\in P^-$ with~$\rho(u)(w+\mathfrak{u}_{k+1})\notin V_a^++V_a^0$,
which gives~$\operatorname{Ad}_u(w)\notin\mathfrak g^+_a+\mathfrak g^0_a$
as required.

By the above we can again apply Theorem~\ref{low entropy theorem} to the
measure~$\mu$ on~${X}$, the subgroup~$U=P^-$ (or~$U=P^+$)
and the element~$a$. Exactly the same argument as in \S\ref{not transient holds} shows the first possible conclusion of Theorem~\ref{low entropy theorem}
leads to unipotent invariance of~$\mu$, hence to prove Proposition~\ref{unipotent invariance proposition 2} we only need to rule out that~\ref{aligned item} of Theorem~\ref{low entropy theorem} holds.

\subsection{Ruling out the aligned case in Theorem~\ref{low entropy theorem}}

Recall that we may assume that~$P^{[\alpha_0]}$ commutes with~$P^-$ and~$P^+$. 
Let us also assume that the above application of Theorem~\ref{low entropy theorem} leads
to the second (aligned) case: I.e.~$\mu$ satisfies that for every $\epsilon > 0$ and compact neighborhood $O$ of the identity in $C_G(\langle P^-,P^+\rangle)$ there exists $B \subset X$ with $\mu (B) > 1 - \epsilon$ and some $\delta > 0$ such that $x, y \in B$
with distance $d (x, y) < \delta$,~$\mu_{x,P^-} = \mu_{y,P^-}$, and~$\mu_{x,P^+} = \mu_{y,P^+}$
 implies that $y = h.x$ for some $h \in O$. Let~$\epsilon=1/10$ and choose~$B$ as above
such that all a.e.-properties of the leafwise measures listed in \S\ref{basic properties} hold on~$B$. 

We now apply recurrence in the direction of~$U=P^{[\alpha_0]}$ as in Proposition~\ref{prop-faithful}
for~$V=P^{[\alpha_0]}\cap U_{\textrm{rad}}$. It follows that there exists some~$x\in B$
and some sequence~$u_k\in P^{[\alpha_0]}$ with~$\pi(u_k)\to\infty$ as~$k\to\infty$
such that~$u_k. x\in B$ has distance less than~$\delta$ from $x$ for all~$k\geq 1$. 
By~\eqref{mu product corollary} we get~$\mu_{x,P^-}=\mu_{u_k.x,P^-}$ (and similarly for~$P^+$)
so that~$u_k.x=h_k.x$ for some bounded sequence~$h_k\in C_G(\langle P^-,P^+\rangle)$. Going
back to~$G$ this becomes~$gu_k^{-1}=\gamma_k g h_k^{-1}$ for some sequence~$\gamma_k\in\Gamma$. 
Then~$\pi(g)\pi(u_k)^{-1}=\pi(\gamma_k)\pi(g)\pi(h_k)^{-1}$ and since~$\pi(u_k)\to\infty$
we see that one of the component of~$\pi(u_k)$, say the component in the direct factor~$G_i$, is unbounded. Let $\pi'$ denote the projection from $G$ to $G_i$. Then as~$\pi'(u_k)\to\infty$ and $\pi'(h_k)$ is bounded, $\pi'(\gamma_k)\to\infty$. On the other hand,
we can again write~$\fl_i$ for the Lie algebra for the Auslander ideal of~$\mathfrak g_i$ defined by
the projection of~$a$ to $G_i$
and note that~$\pi'(u_k),\pi'(h_k)\in C_{G_i}(\fl_i)$. By Lemma~\ref{simple one},
the projection $\pi'(\gamma_k)$ to $G_i$ is central, in contradiction.

\subsection{Finishing the proof of Theorem~\ref{thm:unipotent}}

By Proposition~\ref{unipotent invariance proposition 2} the measure~$\mu$ is invariant under a Zariski connected~$A$-normalized unipotent subgroup~$H_u$ of~$G$.
Let~$H=\langle A, H_u\rangle$ and apply Theorem~\ref{MTclassA} to~$\mu$. Hence there exists
some connected algebraic~$\Q$-group~$\MM$ and a finite index subgroup~$M\leq\MM(\Q_S)$
such that~$\mu$ is invariant under~$g^{-1}Mg\geq H_u$ and supported on the orbit~$\Gamma N_{G}^1(\MM)g$
for some~$g\in G$. Without loss of generality we may assume that~$g=e$. 
Since~$\mu$ is supported on~$\Gamma N_G(\MM)$ and~$\pi(\mu)$ is the Haar measure on the quotient of~$L$,
we see that~$\pi(N_\GG(\MM))=\LL$ or equivalently that~$N_\GG(\MM)=\LL'\UU'$ for some Levy subgroup~$\LL'$
which is isomorphic to~$\LL$ and some~$\UU'<\UU_{\textrm{rad}}$. So let us replace~$\GG$ by~$N_{\GG}(\MM)$
and assume~$\MM\unlhd\GG$.

 If~$\MM_{\textrm{rad}}=\MM\cap\UU_{\textrm{rad}}\unlhd\GG$ is nontrivial, we can take the quotient
by its~$\Q_S$-points (without losing information about~$\mu$ as~$\mu$ is invariant under it) and continue
by induction. So assume that~$\MM_{\textrm{rad}}$ is trivial. In this case~$\UU_{\textrm{rad}}$ commutes with~$\MM$ which shows
that~$\GG$ is isomorphic to~$\MM\times(\tilde\LL\ltimes\UU_{\textrm{rad}})$.
However, this shows that~$X=X_{\MM}\times \tilde X$ where~$X_\MM$ is the homogeneous space corresponding to~$\MM$
and~$\tilde X$ is the one for~$\tilde\LL\ltimes\UU_{\textrm{rad}}$.
Finally,~$\mu$ is invariant under a finite index subgroup of~$\MM(\Q_S)$ which allows us to conclude that~$\mu=m_{X_{\MM}}\times\tilde\mu$ and use induction.

\subsection{Proof of Theorem~\ref{thm:perfect}}
Theorem~\ref{thm:perfect} is a fairly straightforward consequence of Theorem~\ref{higher rank} and Theorem~\ref{thm:unipotent}. For completeness, we provide the details here. Let $\mathbb G _ i = \mathbb L _ i \ltimes \mathbb U _ i$ and $G_i=L_i \ltimes U_i$ be the Levi decomposition of $\mathbb G _i$ for $i = 1, \dots, r$ on the level of algebraic groups and $\Q_S$-points respectively. Let $\mathbb L = \prod_ {i = 1} ^ r \mathbb L _ i$, \ $L = \prod_ {i = 1} ^ r L _ i$ and let $\pi: \mathbb G \to \mathbb L$ be the natural map. We will use the same notation $\pi$ to denote the corresponding maps from the various objects derived from these algebraic groups, 
in particular for the natural map from the $S$-arithmetic quotients $\Gamma \backslash \mathbb G (\Q _ S)$ to $X _ {\mathbb L} = \pi (\Gamma) \backslash L$.

The measure $\pi _*\mu$ is an ergodic joining of the $S$-arithmetic quotients $\pi (\Gamma _ i) \backslash L_i$ and hence by Theorem~\ref{higher rank} it is algebraic, hence is the normalized Haar measure on a single orbit $\pi (\Gamma) H g$ for a finite index subgroup $H$ of $\mathbb H (\Q _ S)$ 
for a semisimple $\Q$-group $\mathbb H$.

It follows that $\mu$ is supported on an orbit of $G' = H \ltimes \left (\prod_ {i = 1} ^ r U _ i \right)$. Because of the joining assumption, and since every $G _ i$ is perfect, $G'$ is perfect, and Theorem~\ref{thm:unipotent} is applicable. Theorem~\ref{thm:perfect} follows.

\section{Arithmetic consequences of the joining theorem}\label{arithmetic section}

For the proof of Theorem~\ref{arithmetic theorem} the following lemmas will be useful:

\begin{lemma}\label{EMS-distortion}
	Let~$s\in\{\infty,$ primes$\}$. 
	Let~$\mathbf G$ be a semisimple algebraic group defined over~$\mathbb Q_s$. Let~$\check \chi$ be a cocharacter of $\mathbf G$ and $\underline{t}$ the corresponding element of $\mathfrak g$ given by $\underline{t}=\frac{d}{d\tau}\check{\chi}(\tau)\bigr|_{\tau=1}$. Let~$g_{\alpha}\in\mathbf G(\mathbb Q_s)$
	be a sequence such that~$\operatorname{Ad}_{g_{\alpha}}(\underline{t})$ is bounded. Then there exists
	a bounded sequence~$h_{\alpha}$ with~$\operatorname{Ad}_{h_{\alpha}}(\underline{t})=\operatorname{Ad}_{g_{\alpha}}(\underline{t})$.
\end{lemma}

\begin{proof}
 Fix some maximal~$\mathbb Q_s$-split torus~$\mathbf T<\mathbf G$ defined over~$\mathbb Q_s$ containing the image of $\check\chi$. Let~$G=KZN$ be the Iwasawa decomposition of~$G$ (see e.g.~\cite[\S3.3.2]{tits}), where~$Z$ is the group of~$\mathbb Q_s$-points
of the centralizer of~$\mathbf T$
and the unipotent group~$N$ is normalized by~$Z$. 
Recall that the choice of~$N$ corresponds to a choice of a positive Weyl chamber in~$\mathbf T$. 
We may assume that~$\check\chi$ belongs to the closure of the chosen positive Weyl chamber. 
It will be more convenient to switch the order of~$Z$ and~$N$ and write~$g_{\alpha}=k_{\alpha}n_{\alpha}a_{\alpha}\in G=KNZ$.

As~$\underline{t}$ is semisimple and normalizes~$N$ we may write the Lie algebra~$\mathfrak n$ of~$N$
as a direct sum of the subalgebra~$\{\underline{n}:[\underline{t},\underline{n}]=0\}$ and 
an~$\operatorname{ad}_{\underline{t}}$-invariant
subalgebra~$V_+<\mathfrak n$ on which all eigenvalues of~$\operatorname{ad}_{\underline{t}}$ are nonzero. Note that as our Weyl chamber was chosen so as to contain $\check\chi$ in its closure, $V^+$ is the Lie algebra of the contracting horospherical subgroup of $\check\chi(\tau)$ for $|\tau|<1$; in particular $V^+$ is invariant under $\operatorname{Ad}_{\check \chi(\tau)}$ for all $\tau \in \Q_s^*$, and is indeed a Lie subalgebra.

As~$a_\alpha$ commutes with~$\underline{t}$ and~$k_\alpha\in K$ remains in a compact set,
it is enough to consider the case~$g_\alpha=n_\alpha\in N$.
As~$N$ is nilpotent
the logarithm map and exponential map are well defined on~$N$ and its Lie algebra.
In particular we can decompose~$n_{\alpha}=n_{\alpha}'c_{\alpha}$ with~$v_{\alpha}=\log n_{\alpha}'\in V_+$
and~$c_{\alpha}$ commuting with~$\underline{t}$.
We claim that a bound on~$\operatorname{Ad}_{n'_{\alpha}}(\underline{t})$
implies that~$n_{\alpha}'$ is bounded, i.e. that the map $n \mapsto \operatorname{Ad}_{n}(\underline{t})$ from $\exp(V_+)$ to $\mathfrak g$ is proper. 
As $\operatorname{Ad}_{n_{\alpha}}(\underline{t})=\operatorname{Ad}_{n'_{\alpha}}(\underline{t})$, once the claim has been established the lemma will follow.

To conclude the proof, we show by induction on nilpotency rank that for any subalgebra $V <V_+$ the map $n \mapsto \operatorname{Ad}_{n}(\underline{t})$ is proper, the base of the induction being the (trivial) case of the trivial algebra $V={0}$.

Suppose that~$\operatorname{Ad}_{n'_{\alpha}}(\underline{t})$ is bounded for a sequence $n_\alpha' = \exp(v_\alpha) \in \exp(V)$. 
Then
\begin{equation}\label{eq:derivatives}
\begin{aligned}
 \operatorname{Ad}_{n_{\alpha}'}(\underline{t})&=\left.\frac{d}{d\tau}n_{\alpha}'\check{\chi}(\tau)(n_{\alpha}')^{-1}\right|_{\tau=1}\\
 &=
 \left.\frac{d}{d\tau}
\left(\check{\chi}(\tau)\exp(\operatorname{Ad}_{\check{\chi}(\tau^{-1})}(v_{\alpha}))\exp(-v_{\alpha})\right)\right|_{\tau=1}\\
&=\underline{t}-\left.\frac{d}{d\tau}
\left(\exp(\operatorname{Ad}_{\check{\chi}(\tau)}(v_{\alpha}))\exp(-v_{\alpha})\right)\right|_{\tau=1}\\
&=\underline{t}-\left.\frac{d}{d\tau}
\log\left(\exp(\operatorname{Ad}_{\check{\chi}(\tau)}(v_{\alpha}))\exp(-v_{\alpha})\right)\right|_{\tau=1}.
\end{aligned}
\end{equation}
However, by the Baker-Campell-Hausdorff formula
\[
\log \left(\exp(\operatorname{Ad}_{\check{\chi}(\tau)}(v_{\alpha}))\exp(-v_{\alpha})\right) \in \operatorname{Ad}_{\check{\chi}(\tau)}(v_{\alpha})-
v_{\alpha}+V_1
\]
where~$V_1=[V,V]$. By definition of~$V_+$ the eigenvalues of $\underline{t}=\frac {d}{d\tau}\check{\chi}(\tau)|_{\tau=1}$
restricted to~$V_+$
are not equal to~$1$. Since by assumption $\operatorname{Ad}_{n_{\alpha}'}(\underline{t})$ remains bounded, by \eqref{eq:derivatives} so must 
the sequence~$v_{\alpha}+V_1$ when considered as a sequence of elements in the vector space~$V/V_1$, i.e.\ that there is a bounded sequence $v'_\alpha \in V$ so that $v_{\alpha}-v'_{\alpha} \in V_1$.  Let $n''_\alpha = \exp(-v_{\alpha}')n_{\alpha}'$.
As both $v'_\alpha$ and~$\operatorname{Ad}_{n_{\alpha}'}(\underline{t})$ are bounded, we have that $\operatorname{Ad}_{n''_{\alpha}}(\underline{t})$
is bounded. By the Baker-Campell-Hausdorff formula 
\(
\log(n_{\alpha}'')\in V_1
\).
By induction, the map $n \mapsto \operatorname{Ad}_{n}(\underline{t})$ restricted to the subgroup~$\exp(V_1)$ is proper, hence the sequences $n_{\alpha}''$ and $ n_{\alpha}'=\exp(v_{\alpha}')n''_\alpha$ are bounded.
\end{proof}

\begin{lemma}\label{unipotent limit} 
 Let~$s\in\{\infty,$ primes$\}$ and let~$\underline{t}\in\mathfrak{sl}_N(\Q_s)$.
 Suppose~$h_\alpha\in\SL_N(\Q_s)$ satisfy that $\operatorname{Ad}_{h_\alpha}(\underline{t})\to\infty$.
 Then there exists a subsequence~$\alpha_j$ of~$\alpha$
 and a sequence~$\lambda_{\alpha_j}\in\Q_s$ converging to zero such that
 $\operatorname{Ad}_{h_{\alpha_j}}(\lambda_{\alpha_j}\underline{t})$ converge to a nonzero nilpotent element $\underline{v}\in\mathfrak{sl}_N(\Q_s)$.
Moreover, for any~$\kappa\in\Q_s$
\[
  \lim_{j\to\infty}h_{\alpha_j}\exp(\kappa\lambda_{\alpha_j}\underline{t})h_{\alpha_j}^{-1}=\exp(\kappa \underline{v}).
\]
\end{lemma}

\begin{proof}
	 If~$s=\infty$ we set~$\lambda_\alpha=\|\operatorname{Ad}_{h_\alpha}(\underline{t})\|^{-1}\to 0$
	for some norm on~$\mathfrak{sl}_N(\R)$. If~$s$ equals a prime we set~$\lambda_\alpha=\|\operatorname{Ad}_{h_\alpha}(\underline{t})\|$ 
	where we use some maximum norm on~$\mathfrak{sl}_n(\Q_s)$ derived from the~$p$-adic norm. Note
	that once more~$\lambda_\alpha\to 0$ and~$\|\lambda_\alpha\operatorname{Ad}_{h_\alpha}(\underline{t})\|=1$
	so that $\lambda_\alpha\operatorname{Ad}_{h_\alpha}(\underline{t})$
	again belongs to a compact subset not containing the zero vector. 
	
	Applying compactness we can choose
	in any case a converging subsequence, dropping the subscript 
	we write~$\underline{v}=\lim_{\alpha\to\infty}\lambda_\alpha\operatorname{Ad}_{h_\alpha}(\underline{t})$.
  Since~$\lambda_\alpha\to0$ and the adjoint representation is not changing the eigenvalues,
the eigenvalues of~$\lambda_\alpha\operatorname{Ad}_{h_\alpha}(\underline{t})$ go to zero
and we obtain that~$\underline{v}$ is nilpotent.

Let~$D\subset\mathfrak{sl}_N(\Q_s)$ be an open set on which the exponential map is defined
and satisfies the usual properties. 
Let~$\kappa\in \Q_s$ and note that~$\kappa\lambda_\alpha\underline{t}\in D$ for all large enough~$\alpha$.
Choose some~$g\in\SL_N(\Q_s)$ such that~$\operatorname{Ad}_g(\kappa \underline{v})\in D$,
which also implies~$\operatorname{Ad}_{gh_\alpha}(\kappa\lambda_\alpha\underline{t})\in D$
for large enough $\alpha$. Applying the exponential map to the latter sequence and taking
the limit we obtain
\begin{multline*}
\lim_{\alpha\to\infty}gh_{\alpha}\exp((\kappa\lambda_\alpha\underline{t})(gh_\alpha)^{-1}=
 \lim_{\alpha\to\infty}\exp(\operatorname{Ad}_{gh_\alpha}(\kappa\lambda_\alpha\underline{t}))=\\
 \exp(\operatorname{Ad}_g(\kappa \underline{v}))=g\exp(\kappa\underline{v})g^{-1},
\end{multline*}
where~$\exp(\kappa\underline{v})$ is defined since~$\underline{v}$ is nilpotent.
Conjugating by~$g^{-1}$ gives the lemma.
\end{proof}

\begin{lemma}\label{local-torus} 
	Let~$\mathbf T _\alpha<\SL(N)$ be a sequence of torus subgroups defined over the~$\Q_s$ for~$s\in\{\infty,$ primes$\}$
	by polynomials of degree~$\leq N$ with $\textup{$\Q_s$-rank}(\mathbf T_\alpha)\geq 1$. Then there exists a~$\Q_s$-split algebraic torus~$\mathbf T_{\text{split}}$
 which is conjugate to the $\Q_s$-split part of~$\mathbf T_\alpha$ for infinitely many~$\alpha$.	
\end{lemma}

\begin{proof}
	Recall that every torus subgroup defined over~$\Q_s$ can be embedded into a maximal torus defined over~$\Q_s$
	and that up to conjugation there are only finitely many maximal torus subgroups of~$\SL(N)$ over $\Q_s$. Therefore, we may choose a subsequence
	and assume that~$\mathbf T_\alpha<\mathbf T_{\text{max}}$ for some maximal $\Q_s$-torus~$\T_{\text{max}}$.
	
	Let~$K$ be the splitting field of~$\mathbf T_{\text{max}}$.
	Given the bound on the degree it is possible to find for every~$\alpha$ finitely many co-characters defined over~$K$
	and with bounded degree that generate~$\mathbf T_\alpha$. However, this implies that there are only finitely many
	possibilities for~$\mathbf T_\alpha$ and the lemma follows.
\end{proof}

\begin{proof}[Proof of Theorem~\ref{arithmetic theorem}]	
	\emph{Reduction to an~$S$-arithmetic space.}
	Let $\mathbb G _ 1, \ldots, \mathbb G _ r$ and~$S_0$ be as in the theorem. 
	We define~$\mathbb G=\mathbb G_1\times\cdots\times\mathbb G_r$ so that~$X_{\mathbb A}=\GG(\Q)\backslash\GG(\mathbb A)$.	
	As is well known (e.g.\ see~\cite{Bo-adele}) 
	there exists a finite set $\Xi\subset\GG(\mathbb A)$ 
	so that
	\[
	\GG(\mathbb A)=\bigsqcup_{\xi\in\Xi}\GG( \Q)\xi G_{S_0\cup\{\infty\}} K_{\neg S_0},
	\]
	where~$G_{S_0\cup\{\infty\}}=\GG(\Q_{S_0\cup\{\infty\}})$ and~$K_{\neg S_0}$ is the compact subgroup~$\prod_{p\notin S_0}\GG(\mathbb Z_p)$. 
	We fix some finite set~$S\supset S_0$
	of places containing~$\infty$ with~$\Xi\subset G_SK_{\neg S}$, where~$G_S=\GG(\Q_S)$
	and~$K_{\neg S}=\prod_{p\notin S}\GG(\mathbb Z_p)$.  Therefore
	\[
	 \GG(\mathbb A)=\GG(\Q)G_SK_{\neg S},
	\]
	and taking the quotient by the rational lattice on the left and the compact group~$K_{\neg S}$ on the right we get the isomorphism
	\[
	 X_S=\Gamma\backslash G_S\cong \GG(\Q)\backslash\GG(\mathbb A)/K_{\neg S}=X_{\mathbb A}/K_{\neg S},
	\]
	where~$\Gamma=\GG(\Q)\cap K_{\neg S}$.

	Let~$G_S^+ \triangleleft G_S$ be the image of the~$\Q_S$-points of the simply connected cover of~$\GG$ in~$G_S$ (including the compact places)
	and recall that~$[G_S:G_S^+]<\infty$.  Hence,
	\[
	 X_S=\Gamma \backslash G_S=\bigsqcup_{j=1}^\ell\Gamma p_jG_S^+=\bigsqcup_{j=1}^\ell\Gamma  G_S^+p_j
	\]
	for some~$p_1,\ldots,p_\ell\in G_S$. 

Let  $\mu_{\alpha,\A}$
denote the push forward of the measure $m_\alpha$ as in the theorem to~$X_{\A}$. 
We may assume that the sequence~$\mu_{\alpha,\A}$ converges to a measure~$\mu_{\A}$ on~$X_{\A}$.
Let~$\mu_{\alpha,S}$ and~$\mu_{S}$ denote the push forward of~$\mu_{\alpha,\A}$ and~$\mu_{\A}$
to~$X_S$. 

Let~$\pi_k:X_{S}\to \Gamma_k\backslash\GG_k(\Q_S)$ be the canonical
projection map, where~$\Gamma_k=\Gamma\cap \GG_k(\Q_S)$. 
Using the assumption of the theorem that~$(\pi_k)_*\mu_{\alpha,S}$ converges to a probability measure
for all~$k=1,\ldots,r$ we see that~$\mu_{S}$ is also a probability measure
and that~$(\pi_k)_*\mu_{\alpha,S}$ converges to~$(\pi_k)_*\mu_{S}$ for~$k=1,\ldots,r$.
By assumption~$(\pi_k)_*\mu_{S}$ is invariant under a finite index subgroup of~$\GG_k(\Q_s)$
for every place~$s$ of~$\Q$. Let~$\tilde{\GG}_k$ be the simply connected cover
of~$\GG_k$ and let~$G_{k,S}^+$ be the image of the~$\Q_S$-points of~$\tilde{\GG}_k$
in~$G_{k,S}=\GG_k(\Q_S)$. Recall that the non-compact almost direct factors of~$G_{k,S}$
have no finite index subgroups (as they are generated by their unipotent one-parameter
subgroups) and that~$\GG_k(\Q_{S_0})$ is non-compact
by the uniform rank assumption of the theorem. 
Using the definition of~$\Gamma_k=\GG_k(\Q)\cap K_{\neg S}$
and strong approximation for the group~$\tilde\GG_k$ it follows that~$(\pi_k)_*\mu_{S}$
is in fact invariant under~$G_{k,S}^+$ for all~$k=1,\ldots,r$.
In particular,~$\Gamma_k\backslash G_{k,S}^+$ is saturated by unipotents.

Fixing some~$j$ with~$\mu_S(\Gamma G_S^+ p_j )>0$ we let~$\mu^+$ denote the restriction
of~$\mu_S$ to~$\Gamma  G_S^+ p_j$. We will show below that~$\mu^+$ is invariant under~$G_S^+$. As this will hold for all~$j$
this will imply that also~$\mu_S$ is invariant under~$G_S^+$. Since~$G_S^+\triangleleft G_S$, by replacing $\mu^+$ with $p_j.\mu^+$ and modify~$g_k^{(\alpha)}$
accordingly we may assume that~$\mu^+$ is supported by~$\Gamma G_S^+$. By the above~$(\pi_k)_*\mu^+$
equals a multiple of the uniform measure on~$\Gamma_k\backslash G_{k,S}^+$ for~$k=1,\ldots,r$.
Since~$\Gamma G_S^+$ is open and closed we see that the restriction~$\mu_{\alpha}^+$ of~$\mu_{\alpha,S}$
to~$\Gamma G_S^+$ converges to~$\mu^+$ as~$\alpha\to\infty$. 
\medskip

\emph{Studying the invariance.}
By the assumptions of the theorem~$\mu_{\alpha,S}$ is invariant under
a finite index subgroup of a~$\Q_{S_0}$-split torus of rank at least two. We need to analyse
this more carefully to see what this
implies for the measure~$\mu^+$.

Recall~$S_0$ is finite and the~$\Q_{S_0}$-rank of~$\T_\alpha$ is at least two. Passing to a subsequence if necessary we may assume that for any $s \in S_0$ the $\Q_s$-rank of $\T_\alpha$ is independent of $\alpha$. Pick $s_1$ so that $r_1=\textup{$\Q_{s_1}$-rank}(\T_\alpha)\geq 1$. If $r_1>1$ set $s_2=s_1$ otherwise choose $s_2\in S_0 \setminus\{ s_1\}$ so that $ \textup{$\Q_{s_2}$-rank}(\T_\alpha)\geq 1$.

Apply Lemma~\ref{local-torus} for the sequence~$\T_\alpha<\SL(N)$
considered over~$\Q_{s_1}$. Therefore, using that the maps~$\iota^\alpha_k$ have bounded degree for~$k=1,\ldots,r$
we may choose another subsequence such that there exist a 
fixed algebraic torus in~$\GG$ of rank $r_1$
defined and split over~$\Q_{s_1}$ that is conjugate over~$\Q_{s_1}$
to a subgroup of the algebraic torus obtained by mapping~$\T_\alpha$ to~$\GG$ via~$(\iota_1^{(\alpha)},\ldots,\iota_r^{(\alpha)})$.
If $r_1=1$, we denote this torus by~$\mathbf T_1$, if $r_1>1$ we choose two different
$\Q_{s_1}$-subtori $\mathbf T_1$ and~$\mathbf T_2$ of this algebraic torus.
If $r_1=1$ we repeat this argument to obtain an additional algebraic one-dimensional torus~$\mathbf T_2<\GG$ defined over~$\Q_{s_2}$ that is conjugate over~$\Q_{s_2}$
to a subgroup of the algebraic torus obtained from~$\T_\alpha$ in~$\GG$. Let $\check{\chi}_1,\check{\chi}_2$ be the cocharacters corresponding to $\mathbf T_1,\mathbf T_2$ respectively.
We note that the assumption that the $\iota^\alpha_k$ have zero dimensional kernel implies that both~$\mathbf T_1$ and~$\mathbf T_2$
have nontrivial projections to~$\GG_k$ for all~$k=1,\ldots,r$.

It follows from our assumption regarding~$\mu_{\alpha,S}$
that there exist finite index subgroups~$A_1<\mathbf T_1(\Q_{s_1})$ and~$A_2<\mathbf T_2(\Q_{s_2})$
such that~$\mu_{\alpha,S}$ is invariant under conjugates of~$A_1$ and~$A_2$. 
Using that~$G_S^+<G_S$ has finite index we replace~$A_j$ by~$A_j\cap G_S^+$ ($j=1,2$)
if necessary
and may assume that~$\mu_{\alpha,S}$ is also invariant under the conjugated group
\[
 I_{\alpha}=h_\alpha A_1A_2 h_\alpha^{-1}
\]
for some~$h_\alpha\in G_S$.

Let~$\underline{t}_j=\left.\frac{d}{d\tau}\check{\chi}_j\right|_{\tau=1}$ for $j=1,2$. 
We now apply Lemma~\ref{EMS-distortion} and see that we have to consider two fundamentally different cases:
Either $\operatorname{Ad}_{h_\alpha}(\underline{t}_1)$ and $\operatorname{Ad}_{h_\alpha}(\underline{t}_2)$ 
both stay bounded along some subsequence of~$\alpha$ or they diverge.

\emph{Bounded case:}
We assume first that the conjugates of~$\underline{t}_1,\underline{t}_2$ stay bounded along
some subsequence of~$\alpha$. In this case we restrict to that subsequence and
apply Lemma~\ref{EMS-distortion}. Hence we may assume that~$h_\alpha$ 
is bounded and converge to some $h_\infty$. 
Therefore the limit measure and its restriction~$\mu^+$ to~$\Gamma G_S^+$
are invariant under~$h_\infty A_1A_2 h_\infty^{-1}$. As~$A_j$ are finite index subgroups of~$\Q_{s_j}$-split torus subgroups
it is easy to choose from this subgroup a class-$\cA'$ subgroup~$A$ isomorphic to~$\Z^2$. 
In other words we have obtained a joining for a~$\Z^2$-action.
Recall that even a single class-$\cA'$ element acts ergodically with respect to the uniform measures on~$\Gamma_k\backslash G_{k,S}^+$
by Proposition~\ref{fancy-Mautner} for all~$k=1,\ldots,r$. This implies that almost every ergodic
component of~$\mu^+$ with respect to the action of~$A$ is an ergodic joining, and therefore these
satisfy all the assumptions of Theorem~\ref{higher rank}.
We conclude that almost every ergodic component of~$\mu^+$ is algebraic over~$\Q$. However, since no two~$\Q$-simple factors of~$\GG$
are isogeneous, this forces the underlying~$\Q$-group in Theorem~\ref{higher rank}
to be~$\GG$. Hence almost every ergodic component must be invariant under a finite index subgroup of~$G_S^+$,
which forces~$\mu^+$ to be a multiple of the homogeneous measure on~$\Gamma G_S^+$.
In particular, it is invariant under~$G_S^+$.

\emph{Unbounded case:}
 Assume now that, e.g.\ the conjugate $\operatorname{Ad}_{h_\alpha}\underline{t}_1$ diverges in norm to infinity.
In this case we apply Lemma~\ref{unipotent limit} and find a nilpotent element~$\underline{v}$
of the Lie algebra and a sequence $\lambda_\alpha \to 0$ in $\Q_{s_1}$ such that
\[
  \lim_{\alpha\to\infty}h_{\alpha}\exp(\kappa\lambda_{\alpha}\underline{t})h_{\alpha}^{-1}=\exp(\kappa \underline{v}) \qquad\text{for any~$\kappa\in\Q_{s_1}$}.
\]

Also recall that~$\mu_{\alpha}$ is invariant under~$I_{\alpha}$ which contains~$h_\alpha A_1 h_\alpha^{-1}$,
which in turn is a finite index subgroup of~$h_\alpha \mathbf T_1(\Q_{s_1}) h_\alpha^{-1}$. Therefore,
there exists a neighborhood~$O$ of~$0$ in~$\Q_{s_1}$ such that~$\exp(\lambda\underline{t}_1)\in A_1$ for all~$\lambda\in O$.
Fix some arbitrary~$\kappa\in\Q_{s_1}$. Then~$\lambda_\alpha \kappa\in O$ for all sufficiently large~$\alpha$,
which gives that~$\mu_{\alpha}$ is invariant under~$h_\alpha \exp(\kappa\lambda_\alpha \underline{t}) h_{\alpha}^{-1}$.
However, as~$\alpha\to\infty$ this converges to~$\exp(\kappa\underline{v})$ and we see that~$\mu_\A$ is invariant under
this element. Since every one-parameter unipotent subgroup of~$G_{s_1}$ belongs to~$G_{s_1}^+$
we see that~$\mu^+$ is invariant under~$U=\exp(\Q_{s_1}\underline{v})$. 
Without loss of generality we may assume that~$U<G_1\times \cdots\times G_\ell$ 
and the projections of~$U$ to~$G_1,\ldots,G_\ell$ are nontrivial, where~$1\leq\ell\leq r$.
We note that~$\pi_k(U)$ acts ergodically with respect to the uniform measure
on~$\Gamma_k\backslash G_k^+$ for~$k=1,\ldots,\ell$ (which follows from Moore's ergodicity theorem and Proposition~\ref{fancy-Mautner}). 

The ergodic components of~$\mu^+$ for the action of~$U$ are naturally identified with
measures on~$(\Gamma_1\times\cdots\times\Gamma_\ell) \backslash (G_1^+\times\cdots\times G_\ell^+)$.
By the erogidicity of the action of~$\pi_k(U)$ for~$k=1,\ldots,\ell$
these ergodic components are ergodic joinings for~$U$. 

We now apply the classification of measures invariant under unipotent one-parameter subgroups (Theorem~\ref{MTclassA})
and obtain that the ergodic components are algebraic measures over~$\Q$. However, the joining property
and the assumption that no two of the factors of~$\GG$ are isogeneous over~$\Q$ implies once more that the
algebraic group in the description of the ergodic components must be~$\GG_1\times\cdots\times\GG_\ell$.
This implies that the ergodic components and so also~$\mu^+$ are invariant under~$G_1^+\times\cdots\times G_\ell^+$.
Considering the quotient of~$\Gamma G_S^+$ by this subgroup we have reduced the number of factors. 
Using induction on the number of factors we conclude that~$\mu^+$ is invariant under~$G_S^+$.

\emph{Returning to the adelic setting:}
We have shown that the measure~$\mu_\A$ on~$X_\A$ obtained by taking the limit of~$\mu_{\alpha,\A}$
projected to the~$S$-arithmetic space~$X_S$ is invariant under~$G_s^+$ for any~$s\in S$ and any sufficiently large set of places~$S$. The union~$\bigcup_SC_c(X_S)$ is dense in $ C_c(X_\A)$ by definition of
the topology of the adeles, and hence~$\mu_\A$ is invariant under~$G_s^+$ for any place~$s\in\{\infty,$ primes$\}$.
Thus~$\mu_\A$ is nearly invariant under~$\GG(\A)$.
\end{proof}

\end{document}